\documentclass[11pt]{article}

\usepackage[margin=1.02in]{geometry}
\usepackage{amsmath,amssymb,amsthm,mathtools}
\usepackage{microtype}
\usepackage{xcolor}
\definecolor{mygreen}{RGB}{0,110,60}
\definecolor{myred}{RGB}{160,30,30}
\usepackage[
 colorlinks=true,
 linkcolor=myred,
 citecolor=mygreen,
 urlcolor=mygreen,
 pdfborder={0 0 0}
]{hyperref}
\usepackage{enumitem}
\usepackage{booktabs}
\usepackage{array}
\usepackage{graphicx}
\usepackage{tikz}
\usetikzlibrary{arrows.meta,positioning,calc,decorations.pathmorphing}

\newtheorem{theorem}{Theorem}[section]
\newtheorem{proposition}[theorem]{Proposition}
\newtheorem{lemma}[theorem]{Lemma}
\newtheorem{corollary}[theorem]{Corollary}
\theoremstyle{definition}
\newtheorem{definition}[theorem]{Definition}
\theoremstyle{remark}
\newtheorem{remark}[theorem]{Remark}

\newcommand{\Z}{\mathbb Z}

\newcommand{\Int}{\operatorname{int}}
\newcommand{\nn}{\langle\!\langle}
\newcommand{\NN}{\rangle\!\rangle}

\title{\textbf{$\mathbb{Z}$-Torus Exteriors and Small Knot-Surgery Four-Manifolds}}
\author{Anar Akhmedov\\
\small School of Mathematics, University of Minnesota\\
\small Visiting Scholar, Department of Mathematics, Harvard University}
\date{}

\begin{document}
\maketitle

\begin{abstract}
We study a cut-and-paste operation in which torus neighborhoods in the
author's symplectic building blocks $Y_K$ and $X_K$, associated to a
genus-one fibered knot $K$, are replaced by marked exteriors of tori in $S^4$ with
infinite cyclic complement group and two null peripheral slopes. We also recall the exact
Luttinger-surgery realization of $M_K\times S^1$ from
$\Sigma_g\times T^2$, keeping the knot-surgery/fiber-sum and
Luttinger-surgery viewpoints in the same framework.

For the trefoil block $Y_K$, two marked replacements give a simply
connected manifold with intersection form $H$, hence a manifold
homeomorphic to $S^2\times S^2$. A one-exterior gluing gives a smooth
homotopy $4$-sphere. For the rank-six construction we use the identity
double of two copies of $Y_K\setminus\nu\Sigma_2$, rather than the
involutive gluing defining the original $X_K$. Two marked replacements along the surviving rim tori give a simply
connected manifold with $e=8$ and $\sigma=0$. An explicit geometric
basis has intersection form $3H$, so the resulting manifold is
homeomorphic to $\#_3(S^2\times S^2)$. The Case II
small-perturbation Seiberg--Witten invariant vanishes. The
Seiberg--Witten invariant of the identity-glued rank-six family also
vanishes; in particular, these manifolds are nonsymplectic.
\end{abstract}

\medskip

\section{Introduction}

Let $C_T=S^4\setminus\Int\nu T$ be the exterior of an embedded torus.
The constructions below use exteriors for which
\[
\pi_1(C_T)\cong\mathbb Z
\]
is generated by a meridian and two primitive boundary classes are null in
$C_T$. After choosing these three classes as a marking of
$\partial C_T\cong T^3$, the two null directions can be matched with
specified generators in a knot-surgery block. The resulting gluing changes
the Euler characteristic, unlike ordinary torus surgery, and at the same
time gives a direct way to control the fundamental group.

The starting blocks are the symplectic manifolds $Y_K$ and $X_K$ introduced
in \cite{Akh07}; see also the related Luttinger-surgery, symplectic-sum,
and knot-surgery constructions with Ozbagci \cite{AO18}, using the knot-surgery framework of Fintushel--Stern
\cite{FS}. For a genus-one fibered knot $K$,
\[
Y_K=(M_K\times S^1)\#_{F=T_m}(M_K\times S^1),
\qquad
X_K=Y_K\#_{\Sigma_2}Y_K,
\]
where $M_K$ is zero surgery on $K$ and
$\Sigma_2\subset Y_K$ is the square-zero genus-two surface obtained from the
unused section and fiber. Section~\ref{sec:YK} recalls the gluing maps
because the peripheral identifications are needed later. It also records the complementary Luttinger-surgery viewpoint:
$M_K\times S^1$ is obtained from $\Sigma_g\times T^2$ by surgeries dictated
by a Dehn-twist factorization of the monodromy of $K$. We then recall
related product-surface Luttinger models starting from
$T^2\times\Sigma_2$ and $\Sigma_2\times\Sigma_2$. Except where an explicit
diffeomorphism is proved, these higher product models are used as comparison
constructions and are not identified with the particular $K$-dependent
manifolds $Y_K$ and $X_K$.

Our main example uses Boyle's turned $1$-twist-spun tori
\cite{Boyle}. Their complements have infinite cyclic fundamental group in
the case used here. Juhasz--Powell \cite[Theorem~1.1]{JP24} proved that
these tori are topologically unknotted; no smooth unknotting result is used.
A second marked exterior comes from a regular fiber of the
Matsumoto--Fukaya achiral torus fibration.

\begin{theorem}\label{thm:intro-main}
Let $K$ be the trefoil. Perform the two marked Boyle-exterior replacements
of Section~\ref{subsec:S2S2case} along the rim tori $R_y$ and $R_d$ in
$Y_K$. The resulting manifold $Z_K(B_1,B_2)$ is simply connected and
\[
Q_{Z_K(B_1,B_2)}\cong H.
\]
Hence
\[
Z_K(B_1,B_2)\approx S^2\times S^2.
\]
The two generators of $H_2(Z_K(B_1,B_2);\mathbb Z)$ are represented by a
square-zero genus-two surface and a square-zero torus meeting once.
\end{theorem}

The proof has two independent parts. The boundary maps reduce the
fundamental group to
\[
\pi_1(Y_K)/\nn d,y\NN,
\]
which is trivial by the presentation recalled from \cite{AP08}. The
surviving genus-two surface and a parallel fiber have intersection matrix
$H$, giving the integral form directly. The argument depends only on the
marked peripheral data; Theorem~\ref{thm:general-replacement} therefore
applies to any two admissibly marked $\mathbb Z$--torus exteriors, including
the Matsumoto--Fukaya fiber exterior.

Two related constructions are included. A single marked exterior glued to
$E_K\times S^1$ gives a smooth homotopy $4$-sphere. For the unknot, with
the geometric marking, this recovers the standard decomposition of $S^4$.
For the rank-six case we do not use the involution in the original
definition of $X_K$. Instead we form the identity double
\[
X_K^{\mathrm{id}}
=
(Y_K\setminus\Int\nu\Sigma_2)
\cup_{\iota}
(Y_K\setminus\Int\nu\Sigma_2),
\]
where $\iota$ is the identity on $\Sigma_2$ and reverses the normal
circle. The two rim tori
\[
R_y=y\times\mu,\qquad R_d=d\times\mu
\]
then survive without any ambiguity. Replacing their neighborhoods by two
admissibly marked exteriors gives a simply connected manifold with
\[
e=8,\qquad \sigma=0.
\]
The geometric basis in Section~\ref{subsec:3S2S2case} survives the
two replacements, so the intersection form is $3H$ and the resulting
manifold is homeomorphic to $\#_3(S^2\times S^2)$.

\section{The symplectic building blocks $M_K\times S^1$, $Y_K$, and $X_K$}\label{sec:YK}

\subsection{From $M_K\times S^1$ to $Y_K$ and $X_K$}
\label{subsec:building-hierarchy}

We follow the construction of \cite[Section~4]{Akh07} exactly. Let $K$ be
a genus-one fibered knot and let $M_K$ be the $3$--manifold obtained by
zero surgery on $K$. Then $M_K\times S^1$ is a torus bundle over a torus.
It contains a symplectic torus fiber $F$ and a symplectic torus section
\[
T_m=m\times S^1,
\]
both of square zero.

The first intermediate manifold is the twisted fiber sum of two copies of
$M_K\times S^1$ in which the \emph{fiber in the first copy} is identified
with the \emph{section in the second copy}:
\[
\boxed{
Y_K=
(M_K\times S^1)_1
\mathop{\#}_{\,F_1=(T_m)_2}
(M_K\times S^1)_2 .
}
\]
This is the notation used in \cite{Akh07},
\[
Y_K=M_K\times S^1\#_{F=T_m}M_K\times S^1.
\]
The gluing is a symplectic normal connected sum in the sense of Gompf and McCarthy--Wolfson \cite{Gompf95,McCarthyWolfson94}, and is written here as a generalized fiber sum in the sense of
\cite[Definition~2.3]{Akh07}, hence is determined by an
orientation-reversing fiber-preserving diffeomorphism of the boundary
three-tori. In the trefoil group calculation of
\cite[Lemma~4.6]{Akh07}, the relevant boundary identifications include
\[
\psi_*(x)=\gamma_1',
\qquad
\psi_*(b)=\gamma_2',
\]
and the remaining boundary circle identifies the knot longitude on one
side with the meridian of the removed torus fiber on the other side.

The two surfaces used to construct the genus-two surface in $Y_K$ are
\emph{not} the two tori used in the fiber sum. Let $T_1$ denote the
section in the first copy of $M_K\times S^1$, and let $T_2$ denote the
fiber in the second copy. These are the complementary section and fiber
left over after the sum. Removing one disk from each and joining their
boundary circles through the fiber-sum neck gives
\[
\Sigma_2=T_1\#T_2\subset Y_K.
\]
$\Sigma_2$ is symplectic and
\[
\Sigma_2^2=0.
\]
The inclusion of its first homology is described by the four classes
\[
(m,x,\gamma_1,\gamma_2).
\]

The second stage takes two copies of this already constructed manifold
$Y_K$. Let
\[
\Sigma_2\subset (Y_K)_1,\qquad
\Sigma_2'\subset (Y_K)_2
\]
be the corresponding genus-two surfaces. The involution
\[
\phi:\Sigma_2'\longrightarrow\Sigma_2
\]
used in \cite{Akh07} acts on first homology by
\[
\boxed{
\phi_*(m')=\gamma_1,\qquad
\phi_*(\gamma_1')=m,\qquad
\phi_*(x')=\gamma_2,\qquad
\phi_*(\gamma_2')=x.
}
\]
The second twisted fiber sum is therefore
\[
\boxed{
X_K=(Y_K)_1\mathop{\#}_{\,\Sigma_2=\Sigma_2',\,\phi}(Y_K)_2 .
}
\]
This is the symplectic manifold denoted $X_K$ in \cite{Akh07}. The later
fundamental-group computation uses the same gluing written in the generators
of the two copies:
\[
y\mapsto e^{-1}f,\qquad
d\mapsto f^{-1}efe^{-1},\qquad
a^{-1}b\mapsto t,\qquad
b^{-1}aba^{-1}\mapsto s;
\]
see also \cite[Lemma~10]{AP08} for the corresponding
fundamental-group formulation.

Thus the construction has two distinct stages:
\[
(M_K\times S^1)_1+(M_K\times S^1)_2
\longrightarrow Y_K,
\qquad
(Y_K)_1+(Y_K)_2
\longrightarrow X_K.
\]
The first sum uses $F_1=(T_m)_2$; the unused $(T_m)_1$ and $F_2$
then combine to form the genus-two surface $\Sigma_2$ used in the second
sum.

\begin{figure}[ht]
\centering
\resizebox{0.94\textwidth}{!}{%
\begin{tikzpicture}[>=Latex,
box/.style={rounded corners=7pt, very thick, minimum width=4.0cm,
      minimum height=1.35cm, align=center},
tag/.style={rounded corners=5pt, very thick, minimum width=4.0cm,
      minimum height=.8cm, align=center},
arr/.style={-{Latex}, very thick},
lab/.style={font=\small,align=center}]
\node[box,draw=mygreen,fill=mygreen!6] (A) at (-5.4,0)
{$\boldsymbol{(M_K\times S^1)_1}$};
\node[box,draw=mygreen,fill=mygreen!6] (B) at (0,0)
{$\boldsymbol{(M_K\times S^1)_2}$};
\node[box,draw=blue!70!black,fill=blue!6] (Y) at (5.4,0)
{$\boldsymbol{Y_K}$};

\node[tag,draw=myred,fill=myred!7] (used) at (-2.7,2.0)
{\textcolor{myred}{$F_1\ \longleftrightarrow\ (T_m)_2$}\\
{\footnotesize fiber of copy 1 \ $\leftrightarrow$ \ section of copy 2}};
\node[font=\small,myred] at (-2.7,2.85) {tori used in the twisted fiber sum};

\node[tag,draw=blue!70!black,fill=blue!6] (T1) at (-5.4,-2.0)
{\textcolor{blue!70!black}{$T_1=(T_m)_1$}\\
{\footnotesize section remaining in copy 1}};
\node[tag,draw=blue!70!black,fill=blue!6] (T2) at (0,-2.0)
{\textcolor{blue!70!black}{$T_2=F_2$}\\
{\footnotesize fiber remaining in copy 2}};
\node[tag,draw=mygreen,fill=mygreen!7] (S) at (5.4,-2.0)
{\textcolor{mygreen}{$\Sigma_2=T_1\#T_2$}\\
{\footnotesize genus two, square zero}};

\draw[very thick,myred] (used.south west) -- ($(A.north)+(0.9,0)$);
\draw[very thick,myred] (used.south east) -- ($(B.north)+(-0.9,0)$);
\draw[arr,blue!70!black] (B.east) -- (Y.west);
\draw[dashed,very thick,blue!70!black] (A.south) -- (T1.north);
\draw[dashed,very thick,blue!70!black] (B.south) -- (T2.north);
\draw[arr,mygreen] (T1.east) -- (T2.west);
\node[font=\small,mygreen] at (-2.7,-2.85)
{join through the fiber-sum neck};
\draw[arr,mygreen] (T2.east) -- (S.west);
\draw[dashed,very thick,mygreen] (S.north) -- (Y.south);
\end{tikzpicture}}
\caption{The first stage of Akhmedov's construction. The fiber $F_1$ in
the first copy is identified with the section $(T_m)_2$ in the second copy.
The section $T_1=(T_m)_1$ remaining in the first copy and the fiber
$T_2=F_2$ remaining in the second copy join through the fiber-sum neck to
form $\Sigma_2=T_1\#T_2\subset Y_K$. Red marks the tori used in the
twisted sum; blue marks the complementary section and fiber; green marks
the resulting genus-two surface.}
\label{fig:YK-correct}
\end{figure}
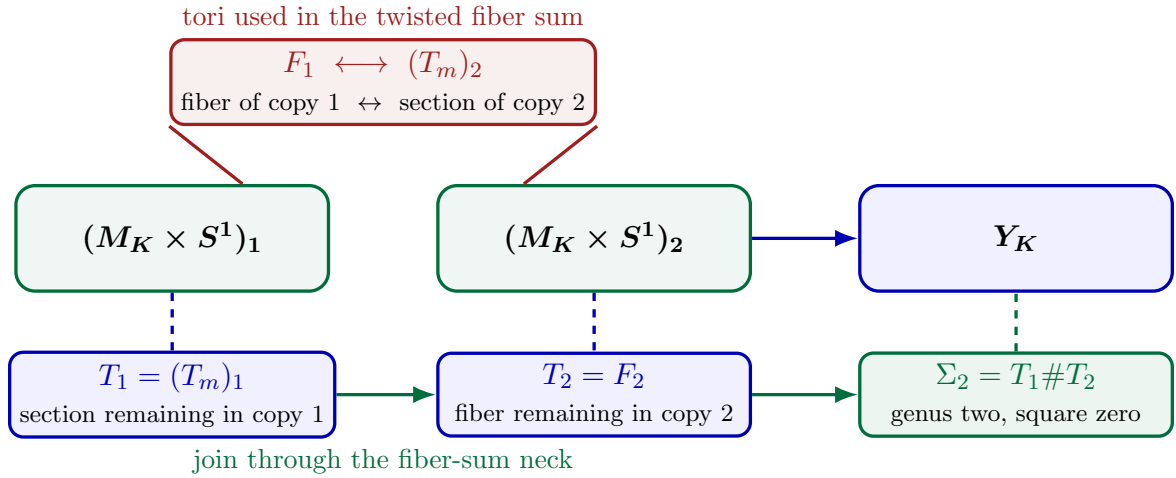

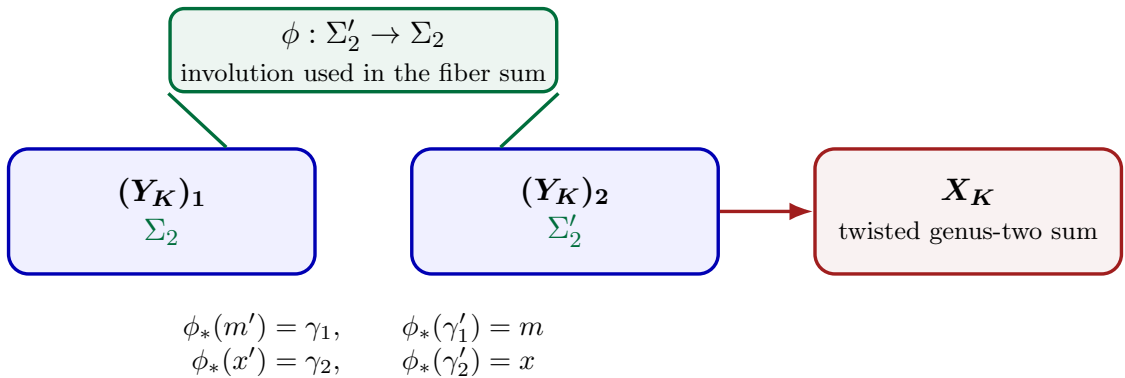
\begin{figure}[ht]
\centering
\resizebox{0.90\textwidth}{!}{%
\begin{tikzpicture}[>=Latex,
box/.style={rounded corners=7pt, very thick, minimum width=3.8cm,
      minimum height=1.55cm, align=center},
arr/.style={-{Latex},very thick},
lab/.style={font=\small,align=center}]
\node[box,draw=blue!70!black,fill=blue!6] (Y1) at (-5.0,0)
{$\boldsymbol{(Y_K)_1}$\\
\textcolor{mygreen}{$\Sigma_2$}};
\node[box,draw=blue!70!black,fill=blue!6] (Y2) at (0,0)
{$\boldsymbol{(Y_K)_2}$\\
\textcolor{mygreen}{$\Sigma_2'$}};
\node[box,draw=myred,fill=myred!6] (X) at (5.0,0)
{$\boldsymbol{X_K}$\\
{\footnotesize twisted genus-two sum}};

\node[rounded corners=5pt,draw=mygreen,very thick,fill=mygreen!7,
  minimum width=4.6cm,minimum height=.9cm,align=center] (phi) at (-2.5,2.0)
{$\phi:\Sigma_2'\to\Sigma_2$\\
{\footnotesize involution used in the fiber sum}};
\draw[very thick,mygreen] (phi.south west) -- ($(Y1.north)+(0.8,0)$);
\draw[very thick,mygreen] (phi.south east) -- ($(Y2.north)+(-0.8,0)$);
\draw[arr,myred] (Y2.east) -- (X.west);

\node[align=center,font=\small] at (-2.5,-1.65)
{$\phi_*(m')=\gamma_1,\qquad \phi_*(\gamma_1')=m$\\
$\phi_*(x')=\gamma_2,\qquad \phi_*(\gamma_2')=x$};
\end{tikzpicture}}
\caption{The second stage. Two copies of $Y_K$ are fiber summed along
their genus-two surfaces using the involution $\phi$ specified in
\cite{Akh07}.}
\label{fig:XK-correct}
\end{figure}

\subsection{$M_K\times S^1$ from Luttinger surgery}
\label{subsec:MK-luttinger}

The basic piece $M_K\times S^1$ has two descriptions that will be used
interchangeably. Besides the mapping-torus description used in
\cite{Akh07,AP08}, its monodromy can be built by Luttinger surgeries on
$\Sigma_g\times T^2$. The latter is a surgery presentation of the same
symplectic manifold, not a new building block.

Let $K$ be a genus-$g$ fibered knot, with fiber $\Sigma_g^1$ and monodromy
$\varphi\in\operatorname{Mod}(\Sigma_g^1)$. After zero surgery on $K$, the
boundary of the fiber is capped by a disk and the monodromy extends to a
diffeomorphism, again denoted $\varphi$, of the closed surface $\Sigma_g$.
Thus
\[
M_K\cong Y(\varphi)
=
[0,1]\times\Sigma_g/
(1,z)\sim(0,\varphi(z)),
\]
and
\[
M_K\times S^1\cong S^1\times Y(\varphi).
\]
For $g\geq1$ this is symplectic by the usual Thurston construction.

The relation with Luttinger surgery is especially direct. Suppose
\[
\varphi=
\tau_{\gamma_r}^{\varepsilon_r}\cdots
\tau_{\gamma_1}^{\varepsilon_1},
\qquad
\varepsilon_i\in\mathbb Z,
\]
is a Dehn-twist factorization. Start with
\[
X_0=\Sigma_g\times S^1_t\times S^1_s.
\]
Choose distinct points $t_1,\ldots,t_r\in S^1_t$ and consider the product
tori
\[
L_i=\gamma_i\times\{t_i\}\times S^1_s.
\]
After a small product perturbation these are Lagrangian, with the product
framing equal to the Lagrangian framing. A unit Luttinger surgery on $L_i$ in the $\gamma_i$-direction changes
the monodromy of the $\Sigma_g$-bundle by a Dehn twist about $\gamma_i$.
Thus a factor $\tau_{\gamma_i}^{\varepsilon_i}$ is realized by
$|\varepsilon_i|$ unit Luttinger surgeries with the sign determined by
$\varepsilon_i$ (or, equivalently, by the corresponding integral
Luttinger-surgery coefficient). This is the mapping-torus example of
Auroux--Donaldson--Katzarkov \cite{ADK}. Carrying out the surgeries in the
order of the chosen factorization gives
\[
\Sigma_g\times T^2
\ \xrightarrow{\ \text{Luttinger surgeries}\ }\ 
S^1\times Y(\varphi)
\cong M_K\times S^1.
\]
Equivalently, a $1/k$ Dehn surgery on a curve in a fiber of a
three-dimensional mapping torus, crossed with the second circle, is the
corresponding Luttinger surgery on the product Lagrangian torus. The
three-dimensional operation cuts the fiber open and reglues it by the
appropriate power of a Dehn twist; crossing with $S^1$ gives the
four-dimensional statement.

The product $\Sigma_g\times T^2$ is symplectic, the $L_i$ are
Lagrangian, and Luttinger surgery preserves the symplectic category
\cite{ADK}. Thus this presentation also recovers the symplectic structure
on $M_K\times S^1$.
The fiber
\[
\Sigma_g\times\{\mathrm{pt}\}\times\{\mathrm{pt}\}
\]
can be chosen away from the surgery tori and survives as a symplectic
surface, while suitable product tori survive as Lagrangian or symplectic
tori according to their position.

\paragraph{The genus-one case.}

For the genus-one fibered knots used in the small-manifold constructions,
the preceding description begins with
\[
T^2\times T^2=T^4.
\]
Let $a,b$ be oriented simple closed curves on the first $T^2$ with
$a\cdot b=1$. For the right-handed trefoil, one may choose conventions so
that the closed-fiber monodromy is represented by
\[
\varphi=\tau_a\tau_b
\]
(up to reversing the order and simultaneously changing the standard
mapping-class conventions). Thus
\[
M_K\times S^1
\]
is obtained from $T^4$ by two Luttinger surgeries, one on a product
Lagrangian torus parallel to
\[
a\times\{\mathrm{pt}\}\times S^1
\]
and the other on a product Lagrangian torus parallel to
\[
b\times\{\mathrm{pt}\}\times S^1,
\]
with the surgery directions $a$ and $b$, respectively. The two tori are
placed over different points of the mapping-torus circle, so they are
disjoint.

For comparison, the figure-eight knot has genus-one monodromy expressible,
with suitable conventions, as a product of one positive and one negative
Dehn twist. Its product with $S^1$ is therefore obtained from $T^4$ by the
same two-torus procedure with one surgery coefficient of the opposite sign.
The distinction between the two knots is encoded in the signs and order of
the Dehn twists, not in a different four-dimensional mechanism.

Thus the mapping-torus and Luttinger-surgery descriptions give the same
symplectic block $M_K\times S^1$. For general genus $g$, the corresponding
starting product is $\Sigma_g\times T^2$. The constructions of
\cite{AP08,AP10,AO18} use this viewpoint to impose fundamental-group relations
by Luttinger surgery.

The fiber sums
\[
Y_K=(M_K\times S^1)\#_{F=T_m}(M_K\times S^1),
\qquad
X_K=Y_K\#_{\Sigma_2}Y_K
\]
are symplectic sums of these blocks \cite{Gompf95,McCarthyWolfson94,Akh07}.

\subsection{Product-surface Luttinger models and the $K$-dependent blocks}
\label{subsec:product-models}

For a fibered knot $K$, the block $M_K\times S^1$ has the exact
Luttinger-surgery realization described above. Related constructions in
\cite{AO18,AS15} start directly from products of surfaces and impose analogous
fundamental-group relations by Luttinger surgery.

The corresponding building-block hierarchy appears in
\cite[Sections~4--5]{Akh07}.  For a genus-one fibered knot, Section~4
constructs $Y_K$ and $X_K=Y_K\#_{\Sigma_2}Y_K$; Section~5 gives the
higher-genus analogues
\[
W_{K,K'}
=
(M_K\times S^1)\#_{F=T_m}(Z_{K'}\times S^1),
\qquad
V_{K,K'}=W_{K,K'}\#_{\Sigma_{g+1}}W_{K,K'},
\]
where $K'$ is a genus-$g$ fibered knot.  The latter satisfy
\[
Q_{V_{K,K'}}\cong (2g-1)H.
\]
Related product-surface Luttinger-surgery models are given in
\cite{AS15}.

\paragraph{(i) The product $T^2\times T^2$.}
For a genus-$g$ fibered knot $K$, the product $M_K\times S^1$ can be
obtained from $\Sigma_g\times T^2$ by the sequence of Luttinger surgeries
corresponding to a chosen Dehn-twist factorization of the monodromy; see
Subsection~\ref{subsec:MK-luttinger}. The number of surgeries is therefore
the length of the chosen factorization. In the genus-one trefoil and
figure-eight cases used below, two surgeries suffice. The $2g$-surgery
families appearing in \cite{AS15} are specific product-surgery models and
are not asserted here for an arbitrary fibered-knot monodromy. In genus one this begins
with
\[
T^2\times T^2=T^4.
\]
If $a,b$ are the standard generators of the first torus and the monodromy is
written, for example, as
\[
\varphi=\tau_a\tau_b
\]
for the right-handed trefoil, then two product Lagrangian tori,
\[
a\times\{t_a\}\times S^1,\qquad
b\times\{t_b\}\times S^1,
\]
with $t_a\neq t_b$, carry the two Luttinger surgeries. The result is
\[
T^4 \ \xrightarrow{\;2\ \text{Luttinger surgeries}\;}\ M_K\times S^1 .
\]

\paragraph{(ii) The product $T^2\times\Sigma_2$.}
Since $T^2\times\Sigma_2\cong\Sigma_2\times T^2$, one can pass directly to a
higher building block without first writing it as a twisted fiber sum. Let
$a_1,b_1,a_2,b_2$ be the standard generators of $\pi_1(\Sigma_2)$ and
$c,d$ those of $\pi_1(T^2)$. The genus-two specialization of the
$\Sigma_n\times T^2$ construction in \cite{AS15} uses four product
Lagrangian tori:
\[
(a_1'\times c',a_1',-1),\qquad
(b_1'\times c'',b_1',-1),
\]
\[
(a_2'\times c',c',+1/p),\qquad
(a_2''\times d',d',+1/q).
\]
For $p=q=1$ all four are Luttinger surgeries. The resulting symplectic manifold is the genus-two product-surgery
model appearing in \cite{AS15}. It is a useful comparison with the earlier
$Y_K/W_{KK'}$ level, but no identification with a particular $Y_K$ is used
here. Thus the fundamental-group relations used in the twisted-fiber-sum block can also be imposed directly on
\[
T^2\times\Sigma_2
\]
by Luttinger surgery on explicit product tori.

\paragraph{(iii) The product $\Sigma_2\times\Sigma_2$.}
At the next level one can start from $\Sigma_2\times\Sigma_2$ itself. Write
$a_i,b_i$ for the standard generators of the first factor and
$c_j,d_j$ for those of the second. The $n=2$ case of the
$\Sigma_n\times\Sigma_2$ construction in \cite{AS15} consists of eight
Luttinger surgeries:
\[
(a_1'\times c_1',a_1',-1),\qquad
(b_1'\times c_1'',b_1',-1),
\]
\[
(a_2'\times c_2',a_2',-1),\qquad
(b_2'\times c_2'',b_2',-1),
\]
\[
(a_2'\times c_1',c_1',+1/p_1),\qquad
(a_2''\times d_1',d_1',+1/q_1),
\]
\[
(a_1'\times c_2',c_2',+1/p_2),\qquad
(a_1''\times d_2',d_2',+1/q_2).
\]
When
\[
p_1=q_1=p_2=q_2=1,
\]
these are all Luttinger surgeries and the result is symplectic. This is the genus-two product-surgery model of \cite{AS15},
parallel to the earlier $X_K/V_{KK'}$ level of the construction. Again,
the present paper uses it as a comparison model rather than as an
unproved identification with a specified $X_K$.
Moreover, the two product surfaces
\[
\Sigma_2\times\{\mathrm{pt}\},
\qquad
\{\mathrm{pt}\}\times\Sigma_2
\]
may be chosen disjoint from the surgery tori; they survive with square zero
and intersection one.

Thus the hierarchy used in this paper admits the following parallel
description:
\[
\begin{array}{ccl}
T^2\times T^2
&\xrightarrow{\text{Luttinger surgeries}}&
M_K\times S^1,\\[1mm]
T^2\times\Sigma_2
&\xrightarrow{\text{Luttinger surgeries}}&
\text{genus-two product-surgery model},\\[1mm]
\Sigma_2\times\Sigma_2
&\xrightarrow{\text{Luttinger surgeries}}&
\text{rank-eight product-surgery model}.
\end{array}
\]
\begin{figure}[ht]
\centering
\resizebox{0.94\textwidth}{!}{%
\begin{tikzpicture}[>=Latex,
box/.style={rounded corners=7pt,very thick,minimum width=4.3cm,
      minimum height=1.35cm,align=center},
arr/.style={-{Latex},very thick},
lab/.style={font=\small,align=center}]
\node[box,draw=mygreen,fill=mygreen!6] (t4) at (-5.6,2.0)
{$T^2\times T^2$};
\node[box,draw=blue!70!black,fill=blue!6] (mk) at (0,2.0)
{$M_K\times S^1$};
\draw[arr,myred] (t4.east) -- node[above,lab,myred]
{$2$ Luttinger surgeries\\($g=1$)} (mk.west);

\node[box,draw=mygreen,fill=mygreen!6] (ts2) at (-5.6,0)
{$T^2\times\Sigma_2$};
\node[box,draw=blue!70!black,fill=blue!6] (yb) at (0,0)
{product-surgery $Y$-model};
\draw[arr,myred] (ts2.east) -- node[above,lab,myred]
{$4$ Luttinger surgeries} (yb.west);

\node[box,draw=mygreen,fill=mygreen!6] (s2s2) at (-5.6,-2.0)
{$\Sigma_2\times\Sigma_2$};
\node[box,draw=blue!70!black,fill=blue!6] (xb) at (0,-2.0)
{product-surgery $X$-model};
\draw[arr,myred] (s2s2.east) -- node[above,lab,myred]
{$8$ Luttinger surgeries} (xb.west);

\node[box,draw=myred,fill=myred!5,minimum width=4.7cm] (old) at (5.6,0)
{knot surgery / twisted\\fiber-sum presentation};
\draw[<->,very thick,mygreen] (mk.east) -- (old.north west);
\draw[<->,very thick,mygreen] (yb.east) -- (old.west);
\draw[<->,very thick,mygreen] (xb.east) -- (old.south west);
\end{tikzpicture}}
\caption{Parallel product-surgery and knot-surgery/fiber-sum descriptions
of the symplectic building blocks. The arrows on the left indicate the
genus-one numbers of Luttinger surgeries. The right-hand box represents the
older knot-surgery and twisted-fiber-sum viewpoint.}
\label{fig:product-luttinger-models}
\end{figure}
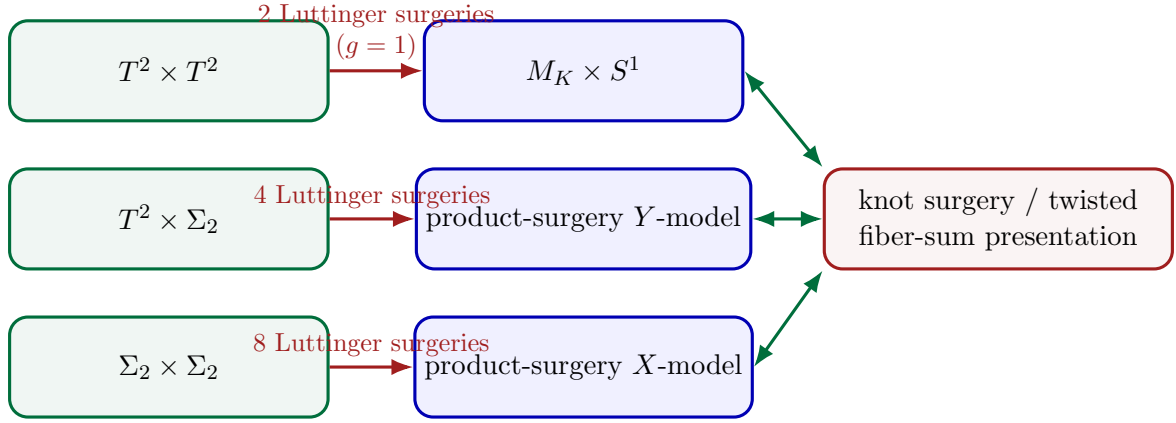

\paragraph{The group of $Y_K$.}
\label{subsec:YK-group}

For the trefoil, the section complement in the first copy is
\[
C_S=(M_K\times S^1)\setminus\Int\nu(b\times x)
\cong (S^3\setminus\Int\nu K)\times S^1,
\]
so
\[
\pi_1(C_S)
=
\langle a,b,x\mid aba=bab,\ [a,x]=[b,x]=1\rangle.
\]
For the second copy, let $C_F$ be the complement of a torus fiber. With
fiber generators $\gamma_1',\gamma_2'$ and base generators $d,y$,
\cite[Lemma~4.5]{Akh07} gives
\[
\pi_1(C_F)=
\langle \gamma_1',\gamma_2',d,y\mid
[\gamma_1',\gamma_2']=[y,\gamma_1']=[y,\gamma_2']=1,
\]
\[
d\gamma_1'd^{-1}=\gamma_1'\gamma_2',
\qquad
d\gamma_2'd^{-1}=(\gamma_1')^{-1}
\rangle .
\]
The twisted fiber sum identifies
\[
x=\gamma_1',\qquad b=\gamma_2',
\]
and the remaining boundary circle gives the longitude--meridian relation.

\begin{proposition}[Akhmedov; Akhmedov--Park]
\label{prop:YK}
The following presentation is \cite[Lemma~4.6]{Akh07}, equivalently
\cite[Lemma~7]{AP08}.
For the trefoil,
\[
\pi_1(Y_K)=
\langle a,b,x,d,y\mid
aba=bab,\ [x,a]=[x,b]=1,
\]
\[
[y,x]=[y,b]=1,\quad
dxd^{-1}=xb,\quad
dbd^{-1}=x^{-1},\quad
ab^2ab^{-4}=[d,y]
\rangle .
\]
\end{proposition}

The genus-two surface $\Sigma_2=T_1\#T_2$ has standard generators whose
images in $\pi_1(Y_K)$ are
\[
a^{-1}b,\qquad b^{-1}aba^{-1},\qquad d,\qquad y,
\]
as in \cite{Akh07}. Its meridian is represented by $[x,b]$ in the
complement calculation used later.

\paragraph{The genus-two surface.}

The surface $\Sigma_2\subset Y_K$ is obtained by joining a punctured fiber in
$C_S$ to a punctured section in $C_F$. As recorded in both \cite{Akh07}
and \cite{AP08}, the standard generators of $\pi_1(\Sigma_2)$ map to
\[
a^{-1}b,\qquad
b^{-1}aba^{-1},\qquad
d,\qquad y.
\]

The complement group is more complicated than $\pi_1(Y_K)$ because the meridian
of $\Sigma_2$ is no longer killed.

\begin{proposition}[Akhmedov--Park]
\label{prop:SigmaComp}
In $Y_K\setminus\Int\nu\Sigma_2$, every meridian of $\Sigma_2$ is conjugate
to the commutator
\[
[x,b].
\]
\end{proposition}

\begin{proof}
This is the meridian statement used in the proof of
\cite[Lemma~10]{AP08}. There the complement is written as
\[
Y_K\setminus\Int\nu\Sigma_2
=
(C_S\setminus\Int\nu F_0)\cup
(C_F\setminus\Int\nu S_0),
\]
where $\Sigma_2$ is the union of a punctured fiber $F_0\subset C_S$ and a
punctured section $S_0\subset C_F$. The proof identifies every meridian with a
conjugate of $[x,b]$, represented by the boundary of a parallel punctured torus
section.
\end{proof}

For the replacement theorem below, however, we use only the following elementary quotient consequence.

\paragraph{Normal generation by $d$ and $y$.}

\begin{lemma}\label{lem:normalYK}
The elements $d$ and $y$ normally generate $\pi_1(Y_K)$:
\[
\pi_1(Y_K)/\nn d,y\NN=1.
\]
\end{lemma}

\begin{proof}
Take the presentation in Proposition~\ref{prop:YK} and impose $d=y=1$.

The relation
\[
dxd^{-1}=xb
\]
becomes
\[
x=xb,
\]
hence
\[
b=1.
\]
Next,
\[
dbd^{-1}=x^{-1}
\]
becomes
\[
1=x^{-1},
\]
so
\[
x=1.
\]
Finally the braid relation
\[
aba=bab
\]
with $b=1$ gives
\[
a=1.
\]
Thus every generator of $\pi_1(Y_K)$ is trivial in the quotient.
\end{proof}

The normal-generation statement is extracted from the presentation in
\cite[Lemma~7]{AP08}; it is not stated there in this form.

\section{Marked $\mathbb Z$--torus exteriors and rim tori}
\label{sec:marked-exteriors}

The replacement argument uses only the peripheral data recorded in the
following definition.

We call a torus exterior $C$ a $\mathbb{Z}$-torus exterior if $\pi_1(C)\cong\mathbb{Z}$ and the meridian of the torus generates this group.

\begin{definition}\label{def:admissible}
Let $T\subset S^4$ be an oriented embedded torus and put
\[
C_T=S^4\setminus\Int\nu T.
\]
An \emph{admissible marking} of $C_T$ is a primitive basis
\[
(\mu_T,\alpha_T,\beta_T)
\]
of $H_1(\partial C_T;\mathbb Z)$ for which
\[
\pi_1(C_T)\cong\mathbb Z\langle\mu_T\rangle,
\qquad
\alpha_T=\beta_T=1\quad\text{in }\pi_1(C_T).
\]
We call $(C_T;\mu_T,\alpha_T,\beta_T)$ an
\emph{admissibly marked $\mathbb Z$--torus exterior}.
\end{definition}

The definition includes more than the abstract condition
$\pi_1(C_T)\cong\mathbb Z$: the two null slopes are part of the data. This
marking is what enters the gluing maps below.

Every automorphism of $H_1(T^3;\mathbb Z)\cong\mathbb Z^3$ is induced by a
diffeomorphism of $T^3$. Thus, whenever we prescribe the image of
one ordered primitive boundary basis in another, the assignment is realized
by a boundary diffeomorphism; by reversing one basis element if necessary,
it may be chosen orientation reversing. Since $\pi_1(T^3)$ is abelian, the
same matrix also determines the induced map on the boundary fundamental
group.

\paragraph{Boyle exteriors.}

Let
\[
B=T_{J,i,1}\subset S^4
\]
be one of Boyle's turned $1$-twist-spun tori and put
\[
C_B=S^4\setminus\Int\nu B.
\]
Boyle's complement-group calculation gives
\[
\pi_1(C_B)\cong\Z.
\]
For a recent discussion of these $j=1$ tori, including the deduction of the
cyclic group from Boyle's formula, see \cite{JP24}, Section 3.

Alexander duality gives
\[
H_1(C_B;\Z)\cong\Z,\qquad
H_2(C_B;\Z)\cong\Z^2,
\]
and therefore
\[
e(C_B)=2.
\]

Choose a boundary basis
\[
(\mu_B,\alpha_B,\beta_B)
\]
for
\[
H_1(\partial C_B)\cong\Z^3
\]
with $\mu_B$ a meridian whose image generates $\pi_1(C_B)$ and with
$\alpha_B,\beta_B$ spanning the kernel of
\[
H_1(\partial C_B)\longrightarrow H_1(C_B).
\]
Because $\pi_1(C_B)\cong\Z$ is abelian, the Hurewicz map
\[
\pi_1(C_B)\longrightarrow H_1(C_B)
\]
is an isomorphism. Consequently
\[
\alpha_B=\beta_B=1
\]
already in $\pi_1(C_B)$.

\begin{lemma}\label{lem:BoyleInv}
For a Boyle $j=1$ exterior,
\[
e(C_B)=2,\qquad \sigma(C_B)=0.
\]
\end{lemma}

\begin{proof}
The Euler characteristic follows from Alexander duality. For the signature, use
\[
S^4=C_B\cup_{\partial}(T^2\times D^2).
\]
The torus neighborhood has zero signature, so Novikov additivity gives
\[
0=\sigma(S^4)=\sigma(C_B).
\]
\end{proof}

\paragraph{The two rim tori.}

Let
\[
\mu=\mu_{\Sigma_2}
\]
denote the circle factor of the boundary
\[
\partial\nu\Sigma_2=\Sigma_2\times S^1_\mu.
\]
By Proposition~\ref{prop:SigmaComp}, the class of $\mu$ in
$\pi_1(Y_K\setminus\nu\Sigma_2)$ is represented, up to conjugacy and orientation,
by $[x,b]$.

Choose embedded representatives of the curves $y,d\subset\Sigma_2$.
These curves meet transversely once on $\Sigma_2$. Their product rim tori
\[
y\times\mu,\qquad d\times\mu
\]
therefore intersect if placed at the same boundary level. To avoid this, choose a
collar
\[
\partial\nu\Sigma_2\times[0,\varepsilon)
\subset
Y_K\setminus\Int\nu\Sigma_2
\]
and choose two distinct numbers $0<t_y<t_d<\varepsilon$. Define
\[
R_y=(y\times\mu)\times\{t_y\},\qquad
R_d=(d\times\mu)\times\{t_d\}.
\]

\begin{lemma}\label{lem:RimGeom}
The tori $R_y$ and $R_d$ are disjoint smoothly embedded tori of self-intersection zero.
\end{lemma}

\begin{proof}
Each is a product of two circles contained in a boundary collar, so each is embedded and
has a product normal framing. Since the two tori occur at different collar parameters,
they are disjoint even though $y$ and $d$ intersect on $\Sigma_2$.
\end{proof}

Let
\[
\rho_y,\rho_d
\]
be the meridians of $R_y,R_d$, respectively, and define
\[
W_K=
Y_K\setminus
\Int(\nu R_y\cup\nu R_d).
\]

On $\partial\nu R_y$ use the ordered basis
\[
(y,\mu_y,\rho_y),
\]
where $\mu_y$ denotes the push-off of the $\mu$-factor at the level $t_y$.
Likewise on $\partial\nu R_d$ use
\[
(d,\mu_d,\rho_d).
\]

We will use the following filling quotient instead of a presentation of $\pi_1(W_K)$.

\begin{lemma}[]\label{lem:standardfill}
There is a natural isomorphism
\[
\pi_1(W_K)/\nn \rho_y,\rho_d\NN
\cong
\pi_1(Y_K).
\]
Under this quotient, the classes $y,d$ map to the elements denoted $y,d$ in
Proposition~\ref{prop:YK}.
\end{lemma}

\begin{proof}
Attach two copies of $T^2\times D^2$ using the original product framings of the deleted
tori. This restores $Y_K$. For each torus attachment, Seifert--van Kampen imposes
exactly the relation that its meridian $\rho$ bounds the $D^2$-factor; the two longitude
directions are identified with the original torus directions. Performing both fillings gives
the displayed quotient.
\end{proof}

\paragraph{Iwase torus fibrations and the Matsumoto--Fukaya regular fiber.}
\label{subsec:iwase}

A second source of admissibly marked exteriors comes from the torus-fibration and
torus-surgery constructions of Iwase. In particular, Iwase studied good torus
fibrations with twin singular fibers and developed Dehn surgery along embedded
$T^2$-knots in $S^4$ \cite{Iwase84,Iwase88}. Larson later gave a systematic
treatment of torus surgery in $S^4$, including handle calculus descriptions of
surgeries on the standard unknotted torus \cite{Larson18}.

The model used here is Matsumoto's genus-one achiral Lefschetz fibration
\[
f_{\mathrm{MF}}:S^4\longrightarrow S^2,
\]
usually called the Matsumoto--Fukaya fibration. It has two critical points
of opposite signs and regular fiber $T^2$; see
\cite{Matsumoto85,DKZ17}. In the standard description of
\cite[Section~2]{DKZ17}, the two vanishing cycles are isotopic, so the
local monodromies are inverse Dehn twists about a primitive curve $c$.

Let
\[
F\subset S^4
\]
be a regular torus fiber and put
\[
C_{\mathrm{MF}}=S^4\setminus\Int\nu F.
\]
Alexander duality gives, independently of the particular torus embedding,
\[
H_1(C_{\mathrm{MF}};\mathbb Z)\cong\mathbb Z,
\qquad
H_2(C_{\mathrm{MF}};\mathbb Z)\cong\mathbb Z^2,
\qquad
e(C_{\mathrm{MF}})=2.
\]
These homological facts agree with the corresponding invariants of a Boyle
exterior. What is not automatic is the fundamental group and the marked
peripheral map
\[
\pi_1(\partial C_{\mathrm{MF}})\longrightarrow \pi_1(C_{\mathrm{MF}}).
\]

\begin{proposition}\label{prop:MF-complement}
Let $F$ be a regular fiber of the Matsumoto--Fukaya fibration
$f_{\mathrm{MF}}:S^4\to S^2$, and put
\[
C_{\mathrm{MF}}=S^4\setminus\Int\nu F.
\]
Then
\[
\pi_1(C_{\mathrm{MF}})\cong\mathbb Z.
\]
Moreover, if $c,\ell$ is a basis of $\pi_1(F)$ with $c$ the common
vanishing cycle and $\mu_F$ is the positively oriented meridian of $F$, then
$c$ is trivial in $\pi_1(C_{\mathrm{MF}})$ and the image of $\ell$ is a
primitive generator. Alexander duality identifies $[\mu_F]$ as a primitive
generator of $H_1(C_{\mathrm{MF}};\mathbb Z)$. Since
$\pi_1(C_{\mathrm{MF}})$ is infinite cyclic, the Hurewicz map is an
isomorphism; hence, after orienting $\ell$ suitably,
\[
\mu_F=\ell^{\varepsilon}\quad\text{in }\pi_1(C_{\mathrm{MF}}),
\qquad \varepsilon\in\{1,-1\}.
\]
Consequently
\[
\alpha_F=c,\qquad
\beta_F=\mu_F\ell^{-\varepsilon}
\]
are trivial in $\pi_1(C_{\mathrm{MF}})$, and
\[
(\mu_F,\alpha_F,\beta_F)
\]
is a primitive basis of $H_1(\partial C_{\mathrm{MF}};\mathbb Z)$.
Hence this basis gives an admissible marking of $C_{\mathrm{MF}}$ in the
sense of Definition~\ref{def:admissible}.
\end{proposition}

\begin{proof}
Choose a small disk $D_0\subset S^2$ containing no critical values and with
$f_{\mathrm{MF}}^{-1}(D_0)=\nu F$. Its complementary disk
\[
D=\overline{S^2\setminus D_0}
\]
contains the two critical values. Hence $C_{\mathrm{MF}}$ is the total space
of a genus-one achiral Lefschetz fibration over $D$ with two critical points.
In the standard Matsumoto--Fukaya model the two vanishing cycles are
isotopic to the same primitive curve $c\subset T^2$, while the critical
points have opposite signs; see \cite{Matsumoto85} and
\cite[Section~2]{DKZ17}. For a Lefschetz fibration over a disk, attaching a Lefschetz
$2$--handle kills its vanishing cycle in the fundamental group; the sign of
the critical point changes the handle framing but not this relation.
Consequently
\[
\pi_1(C_{\mathrm{MF}})
\cong
\pi_1(T^2)/\nn c\NN
\cong
\langle c,\ell\mid[c,\ell]=1,\ c=1\rangle
\cong\mathbb Z\langle\ell\rangle.
\]
We next identify the geometric meridian. For every oriented embedded
torus in $S^4$, Alexander duality gives
\[
H_1(S^4\setminus\Int\nu F;\mathbb Z)\cong\mathbb Z,
\]
generated by a positively oriented meridian $\mu_F$. Since the fundamental
group just computed is already infinite cyclic, the Hurewicz map is an
isomorphism. Thus $\mu_F$ is a generator of
$\pi_1(C_{\mathrm{MF}})$ as well. The element $\ell$ is also a generator. Here $\mu_F$ is the normal
circle in
\[
\partial\nu F=F\times S^1,
\]
hence the geometric meridian used in Alexander duality. Thus, after
choosing the orientation of $\ell$,
\[
\mu_F=\ell^\varepsilon,\qquad \varepsilon=\pm1.
\]
It follows that
\[
c=1,\qquad \mu_F\ell^{-\varepsilon}=1
\]
in the complement group.

Finally, relative to the geometric boundary basis $(c,\ell,\mu_F)$, the
three classes
\[
\mu_F,\qquad c,\qquad \mu_F\ell^{-\varepsilon}
\]
form an integral basis of $H_1(T^3)$: the corresponding change-of-basis
matrix has determinant $\pm1$. This proves the peripheral statement.
\end{proof}

In the ordered boundary basis $(c,\ell,\mu_F)$, the inclusion
\[
i_*:H_1(\partial C_{\mathrm{MF}};\mathbb Z)\longrightarrow
H_1(C_{\mathrm{MF}};\mathbb Z)\cong\mathbb Z\langle u\rangle
\]
is therefore
\[
i_*(c)=0,\qquad i_*(\ell)=u,\qquad
i_*(\mu_F)=\varepsilon u.
\]
Equivalently, its matrix is the row vector
\[
\begin{pmatrix}0&1&\varepsilon\end{pmatrix}.
\]
Thus
\[
\ker i_*=
\mathbb Z\langle c,\mu_F-\varepsilon\ell\rangle,
\]
which is a primitive rank-two direct summand. This is the marked
peripheral information used in the replacement theorem; it is stronger than
the abstract statement $\pi_1(C_{\mathrm{MF}})\cong\mathbb Z$.

\begin{remark}
The theorem determines the topological type only. Since
$C_{\mathrm{MF}}$ comes with an explicit achiral Lefschetz-fibration
handlebody, it gives an explicit test case for the marked handle calculus problem
of determining the smooth type.
\end{remark}

\section{Three small-manifold cases}
\label{sec:threecases}

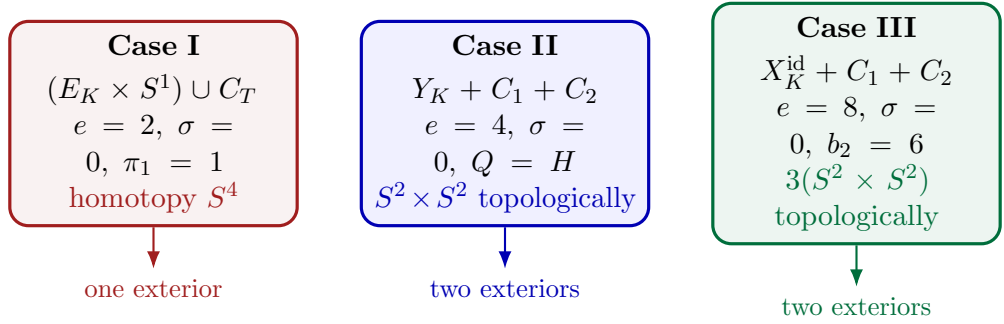
\begin{figure}[ht]
\centering
\begin{tikzpicture}[>=Latex,
case/.style={rounded corners=7pt,very thick,minimum width=3.75cm,
minimum height=2.15cm,text width=3.45cm,align=center,inner sep=5pt},
arr/.style={-{Latex},thick}]
\node[case,draw=myred,fill=myred!6] (s4) at (-4.65,0)
{\textbf{Case I}\\[1mm]
$(E_K\times S^1)\cup C_T$\\
$e=2,\ \sigma=0,\ \pi_1=1$\\
\textcolor{myred}{homotopy $S^4$}};

\node[case,draw=blue!70!black,fill=blue!6] (s2) at (0,0)
{\textbf{Case II}\\[1mm]
$Y_K+C_1+C_2$\\
$e=4,\ \sigma=0,\ Q=H$\\
\textcolor{blue!70!black}{$S^2\times S^2$ topologically}};

\node[case,draw=mygreen,fill=mygreen!7] (s3) at (4.65,0)
{\textbf{Case III}\\[1mm]
$X_K^{\mathrm{id}}+C_1+C_2$\\
$e=8,\ \sigma=0,\ b_2=6$\\
\textcolor{mygreen}{$3(S^2\times S^2)$ topologically}};

\draw[arr,myred] (s4.south) -- ++(0,-0.55)
node[below,font=\small] {one exterior};
\draw[arr,blue!70!black] (s2.south) -- ++(0,-0.55)
node[below,font=\small] {two exteriors};
\draw[arr,mygreen] (s3.south) -- ++(0,-0.55)
node[below,font=\small] {two exteriors};
\end{tikzpicture}
\caption{The three topological targets.  Case III has intersection form $3H$.}
\label{fig:three-cases}
\end{figure}

The applications are grouped by topological target. Boyle exteriors are treated first; the Matsumoto--Fukaya fiber exterior is then substituted using Proposition~\ref{prop:MF-complement}.

\subsection{Case I: homotopy $4$-spheres}
\label{subsec:S4case}

Let $K\subset S^3$ be a knot, let
\[
E_K=S^3\setminus\Int\nu K,
\]
and write
\[
\partial(E_K\times S^1)
=
\partial E_K\times S^1\cong T^3.
\]
We use the ordered peripheral basis
\[
(m_K,\lambda_K,x),
\]
where $m_K$ and $\lambda_K$ are respectively a meridian and the preferred
longitude of $K$, and $x$ is the extra $S^1$-factor. Thus
\[
\pi_1(E_K\times S^1)
\cong \pi_1(E_K)\times\langle x\rangle,
\]
and the knot group $\pi_1(E_K)$ is normally generated by $m_K$.

\paragraph{Boyle exterior.}

Let $(C_B;\mu_B,\alpha_B,\beta_B)$ be an admissibly marked Boyle exterior.
Choose an orientation-reversing boundary diffeomorphism
\[
\Phi_B:\partial(E_K\times S^1)\longrightarrow\partial C_B
\]
whose induced map, up to inverses required by orientation conventions, is
\[
\boxed{
m_K\longmapsto\alpha_B,\qquad
x\longmapsto\beta_B,\qquad
\lambda_K\longmapsto\mu_B.
}
\]
Define
\[
\Sigma_B(K)
=
(E_K\times S^1)\cup_{\Phi_B} C_B.
\]

\begin{theorem}\label{thm:S4}
For every knot $K$,
\[
\pi_1(\Sigma_B(K))=1,\qquad
e(\Sigma_B(K))=2,\qquad
\sigma(\Sigma_B(K))=0.
\]
Consequently $\Sigma_B(K)$ is a smooth homotopy $4$-sphere.
\end{theorem}

\begin{proof}
Set $G_K=\pi_1(E_K)$. Since
\[
\pi_1(C_B)=\langle\mu_B\rangle\cong\mathbb Z,
\qquad
\alpha_B=\beta_B=1
\]
in the Boyle exterior, Seifert--van Kampen gives
\[
\pi_1(\Sigma_B(K))
\cong
\frac{(G_K\times\langle x\rangle)*\langle\mu_B\rangle}
{\nn m_K,\ x,\ \lambda_K\mu_B^{-1}\NN}.
\]
The relation $m_K=1$ kills $G_K$ because a knot group is normally generated
by a meridian. Hence the preferred longitude $\lambda_K$ also becomes
trivial. The relation $\lambda_K=\mu_B$ then kills the generator of the
Boyle exterior, and $x=1$ kills the remaining product factor. Therefore
\[
\pi_1(\Sigma_B(K))=1.
\]

Since $\chi(E_K)=0$ and $\chi(S^1)=0$,
\[
e(E_K\times S^1)=0.
\]
The common boundary is $T^3$, which has Euler characteristic zero, while
$e(C_B)=2$. Hence
\[
e(\Sigma_B(K))=2.
\]
Both pieces have signature zero: $E_K\times S^1$ has trivial intersection
pairing and $\sigma(C_B)=0$ by Lemma~\ref{lem:BoyleInv}. Novikov
additivity therefore gives $\sigma(\Sigma_B(K))=0$.

A closed simply connected $4$-manifold with Euler characteristic two has
$H_2=0$; Poincar\'e duality then shows that it has the integral homology of
$S^4$. Since both spaces are simply connected, a degree-one map
$\Sigma_B(K)\to S^4$ is a homology equivalence and hence, by Whitehead's
theorem, a homotopy equivalence. Thus $\Sigma_B(K)$ is a smooth homotopy
$4$-sphere.
\end{proof}

\begin{corollary}\label{cor:unknot-standard}
For the unknot $U$, if $\Phi_B$ is the geometric marking coming from the
original decomposition
\[
S^4=C_B\cup\nu B,
\]
then
\[
\Sigma_B(U)\cong S^4.
\]
\end{corollary}

\begin{proof}
For the unknot,
\[
E_U\cong S^1\times D^2,
\qquad
E_U\times S^1\cong T^2\times D^2.
\]
Under the geometric peripheral identification, the two core directions are
sent to the two null slopes of $C_B$ and the meridian of the standard
$T^2\times D^2$ is sent to $\mu_B$. Thus the gluing is precisely the
original attachment of $\nu B$ to its exterior.
\end{proof}

\paragraph{Matsumoto--Fukaya exterior.}

The same construction works with the regular-fiber exterior
$C_{\mathrm{MF}}$. Using its admissible marking
$(\mu_F,\alpha_F,\beta_F)$ from Proposition~\ref{prop:MF-complement}, choose
\[
m_K\mapsto\alpha_F,\qquad
x\mapsto\beta_F,\qquad
\lambda_K\mapsto\mu_F,
\]
and put
\[
\Sigma_{\mathrm{MF}}(K)
=
(E_K\times S^1)\cup C_{\mathrm{MF}}.
\]

\begin{corollary}\label{cor:MFsphere}
For every knot $K$, $\Sigma_{\mathrm{MF}}(K)$ is a smooth homotopy $4$-sphere.
\end{corollary}

\begin{proof}
The proof of Theorem~\ref{thm:S4} uses only the defining properties of an
admissibly marked $\mathbb Z$--torus exterior. Proposition~\ref{prop:MF-complement}
shows that $C_{\mathrm{MF}}$ has exactly these properties.
\end{proof}

For nontrivial $K$ we make no claim that either homotopy sphere is smoothly
standard. In the Matsumoto--Fukaya case the explicit achiral Lefschetz
handlebody makes this a concrete marked handle calculus problem.

\subsection{Case II: manifolds homeomorphic to $S^2\times S^2$}
\label{subsec:S2S2case}

We use the trefoil block $Y_K$ and the two disjoint rim tori
$R_y,R_d$ from Section~\ref{sec:marked-exteriors}. Recall that on their
boundaries we use the ordered bases
\[
(y,\mu_y,\rho_y),\qquad (d,\mu_d,\rho_d),
\]
where $\rho_y,\rho_d$ are the rim-torus meridians and $\mu_y,\mu_d$ are
push-offs of the meridian of $\Sigma_2$.

\paragraph{Two Boyle exteriors.}

Let
\[
(C_{B_i};\mu_i,\alpha_i,\beta_i),\qquad i=1,2,
\]
be two marked Boyle exteriors. Choose orientation-reversing boundary
diffeomorphisms so that, up to harmless inverses,
\[
\boxed{
y\mapsto\alpha_1,\qquad
\rho_y\mapsto\beta_1,\qquad
\mu_y\mapsto\mu_1,
}
\]
and
\[
\boxed{
d\mapsto\alpha_2,\qquad
\rho_d\mapsto\beta_2,\qquad
\mu_d\mapsto\mu_2.
}
\]
Define
\[
Z_K(B_1,B_2)
=
W_K\cup C_{B_1}\cup C_{B_2}.
\]

\begin{theorem}\label{thm:BoyleS2S2}
For the trefoil block $Y_K$,
\[
\pi_1\bigl(Z_K(B_1,B_2)\bigr)=1,\qquad
e\bigl(Z_K(B_1,B_2)\bigr)=4,\qquad
\sigma\bigl(Z_K(B_1,B_2)\bigr)=0.
\]
The surviving surfaces $\Sigma_2$ and $F'$ satisfy
\[
\Sigma_2^2=(F')^2=0,\qquad \Sigma_2\cdot F'=1.
\]
Consequently
\[
Q_{Z_K(B_1,B_2)}\cong H
\]
and
\[
Z_K(B_1,B_2)\approx S^2\times S^2.
\]
\end{theorem}

\begin{proof}
In each Boyle exterior, $\alpha_i$ and $\beta_i$ are trivial and $\mu_i$
generates the cyclic fundamental group. Hence the two boundary gluings
impose
\[
y=\rho_y=1,\qquad d=\rho_d=1
\]
in the amalgamated product, while $\mu_y$ and $\mu_d$ are identified with
the two cyclic generators contributed by the Boyle pieces. After imposing
$\rho_y=\rho_d=1$, Lemma~\ref{lem:standardfill} identifies the central
quotient with $\pi_1(Y_K)$. The remaining relations $y=d=1$ therefore give
\[
\pi_1(Y_K)/\nn y,d\NN,
\]
which is trivial by Lemma~\ref{lem:normalYK}. Once the image of the central
group is trivial, the two generators $\mu_1,\mu_2$, being identified with
$\mu_y,\mu_d$, are trivial as well. Thus
\[
\pi_1\bigl(Z_K(B_1,B_2)\bigr)=1.
\]

The original $Y_K$ has
\[
e(Y_K)=0,\qquad \sigma(Y_K)=0
\]
by \cite{Akh07}. Removing two copies of $T^2\times D^2$ changes neither
invariant. Each Boyle exterior has Euler characteristic two and signature
zero by Lemma~\ref{lem:BoyleInv}, so
\[
e\bigl(Z_K(B_1,B_2)\bigr)=4,\qquad
\sigma\bigl(Z_K(B_1,B_2)\bigr)=0.
\]

The rim tori lie in a collar of $\partial\nu\Sigma_2$, hence are disjoint
from the core $\Sigma_2$. Choose the parallel fiber $F'$ so that its unique
intersection point with $\Sigma_2$ is away from the curves $y\cup d$.
Then $F'$ is also disjoint from the two rim-torus neighborhoods. Therefore
both surfaces survive the replacements and retain
\[
\Sigma_2^2=(F')^2=0,\qquad \Sigma_2\cdot F'=1.
\]
Simple connectivity and $e=4$ imply
\[
H_2(Z_K;\mathbb Z)\cong\mathbb Z^2.
\]
The two displayed classes have intersection determinant $-1$, so they form
an integral basis and the intersection form is $H$. Freedman's
classification \cite{Freedman} gives the homeomorphism with $S^2\times S^2$.
\end{proof}

\paragraph{Two Matsumoto--Fukaya exteriors.}

Replace $C_{B_1},C_{B_2}$ by two copies of $C_{\mathrm{MF}}$, using the same
boundary assignments with the admissible marking from
Proposition~\ref{prop:MF-complement}. Denote the result by
$Z_K^{\mathrm{MF}}$.

\begin{corollary}\label{cor:MFS2S2}
For the trefoil block,
\[
\pi_1(Z_K^{\mathrm{MF}})=1,\qquad
Q_{Z_K^{\mathrm{MF}}}\cong H,
\]
and
\[
Z_K^{\mathrm{MF}}\approx S^2\times S^2.
\]
\end{corollary}

\begin{proof}
The proof of Theorem~\ref{thm:BoyleS2S2} uses only the admissible marked
peripheral data and the facts $e(C_T)=2$, $\sigma(C_T)=0$ for a torus
exterior in $S^4$. Proposition~\ref{prop:MF-complement} supplies the
required marking for $C_{\mathrm{MF}}$.
\end{proof}

\begin{theorem}[general marked replacement principle]
\label{thm:general-replacement}
Let $(C_1;\mu_1,\alpha_1,\beta_1)$ and
$(C_2;\mu_2,\alpha_2,\beta_2)$ be any two admissibly marked
$\mathbb Z$--torus exteriors. Glue them to $W_K$ using the same assignments
as above. Then the resulting manifold is simply connected, has intersection
form $H$, and is homeomorphic to $S^2\times S^2$.
\end{theorem}

\begin{proof}
The fundamental-group proof is identical to the first part of
Theorem~\ref{thm:BoyleS2S2}. For any torus $T\subset S^4$, Alexander
duality gives $e(C_T)=2$, and Novikov additivity applied to
\[
S^4=C_T\cup(T^2\times D^2)
\]
gives $\sigma(C_T)=0$. The surfaces $\Sigma_2,F'$ lie in the unchanged
central region and therefore provide the same hyperbolic basis.
\end{proof}

\subsection{Case III: the $3(S^2\times S^2)$ target}
\label{subsec:3S2S2case}

For this case we modify the second fiber-sum step.  Let
\[
Y_i^\circ=Y_K\setminus\Int\nu\Sigma_2,\qquad i=1,2,
\]
and define
\[
X_K^{\mathrm{id}}
=
Y_1^\circ\cup_{\iota}Y_2^\circ,
\]
where
\[
\iota:\Sigma_2\times S^1\longrightarrow\Sigma_2\times S^1
\]
is the identity on $\Sigma_2$ and reverses the normal $S^1$-factor.  Thus
$\iota$ is orientation reversing on the boundary and is the product gluing
on the surface directions.

In a collar of the first boundary component choose two distinct levels and
set
\[
R_y=(y\times\mu)\times\{t_y\},\qquad
R_d=(d\times\mu)\times\{t_d\},
\qquad 0<t_y<t_d.
\]
These are the same product rim tori used in Case II.  Because they lie at
positive collar levels in $Y_1^\circ$, and because the gluing is product on
$\Sigma_2$, they survive as disjoint square-zero tori in
$X_K^{\mathrm{id}}$.  
We first record the effect of the identity gluing on first homology.

\begin{lemma}\label{lem:id-double-H1}
For the trefoil block,
\[
H_1(X_K^{\mathrm{id}};\mathbb Z)
\cong
\mathbb Z\langle d\rangle\oplus
\mathbb Z\langle y\rangle.
\]
Moreover
\[
e(X_K^{\mathrm{id}})=4,\qquad
\sigma(X_K^{\mathrm{id}})=0.
\]
\end{lemma}

\begin{proof}
We first compute the abelianization of
\[
G=\pi_1(Y_K\setminus\Int\nu\Sigma_2)
\]
from \cite[Lemma~4.7]{Akh07}.  In additive notation the braid relation
gives $a=b$.  The relation
\[
dxd^{-1}=xb
\]
then gives $b=0$, hence $a=0$, and
\[
dbd^{-1}=x^{-1}
\]
gives $x=0$.  The meridian $[x,b]$ is a commutator, and the remaining
relations impose no relation on $d$ or $y$.  Since the additional
generators in Lemma~4.7 lie in the normal subgroup generated by $[x,b]$,
they vanish in the abelianization.  Thus
\[
H_1(Y_K\setminus\Int\nu\Sigma_2;\mathbb Z)
\cong
\mathbb Z\langle d\rangle\oplus
\mathbb Z\langle y\rangle.
\]

The four standard generators of $H_1(\Sigma_2)$ map to
\[
a^{-1}b,\qquad b^{-1}aba^{-1},\qquad d,\qquad y
\]
in the first copy \cite[Lemma~4.6]{Akh07}.  After abelianization the first
two vanish and the last two generate.  Under the identity gluing, $d$ is
identified with $d'$ and $y$ with $y'$.  Hence the Mayer--Vietoris map
\[
H_1(\Sigma_2\times S^1)\longrightarrow
H_1(Y_1^\circ)\oplus H_1(Y_2^\circ)
\]
has image generated by
\[
(d,-d'),\qquad (y,-y').
\]
The normal-circle generator maps trivially to both abelianizations because
it is represented by the meridian $[x,b]$.  Therefore the cokernel is
\[
\mathbb Z\langle d\rangle\oplus\mathbb Z\langle y\rangle.
\]

Finally,
\[
e(X_K^{\mathrm{id}})
=
2e(Y_K)-2e(\Sigma_2)
=
4,
\]
and Novikov additivity gives
\[
\sigma(X_K^{\mathrm{id}})=0.
\]
\end{proof}

We next identify the full second homology of the identity double.  Put
\[
\gamma_1=a^{-1}b,\qquad
\gamma_2=b^{-1}aba^{-1}.
\]
With the standard genus-two generators used in
\cite[Lemma~4.6]{Akh07}, the ordered curves
\[
(\gamma_1,\gamma_2,d,y)
\]
form a symplectic basis of $H_1(\Sigma_2)$: after fixing orientations,
\[
\gamma_1\cdot\gamma_2=1,\qquad d\cdot y=1,
\]
and all other pairings vanish.  The first two classes lie in the kernel of
\[
H_1(\Sigma_2)\longrightarrow H_1(Y_K^\circ),
\]
whereas $d,y$ map to the two generators of $H_1(Y_K^\circ)$.

\begin{lemma}[rim tori and dual classes]
\label{lem:rim-dual-lattice}
Let
\[
K=\ker\!\left\{
H_1(\Sigma_2;\mathbb Z)\longrightarrow
H_1(Y_K^\circ;\mathbb Z)
\right\}.
\]
For the trefoil block,
\[
K=\mathbb Z\langle\gamma_1,\gamma_2\rangle,
\qquad
\gamma_1\cdot\gamma_2=1.
\]
In the identity double $X_K^{\mathrm{id}}$, choose the two rim tori
\[
R_i=\gamma_i\times\mu
\]
at distinct levels of the gluing neck.  There are classes
$V_1,V_2\in H_2(X_K^{\mathrm{id}};\mathbb Z)$ such that
\[
R_i^2=0,\qquad
R_i\cdot V_j=\gamma_i\cdot\gamma_j,
\]
and the self-intersections $V_i^2$ are even.  The rank-four lattice
generated by $R_1,R_2,V_1,V_2$ is therefore isomorphic to $2H$.
\end{lemma}

\begin{proof}
Because $\gamma_i$ is null-homologous in $Y_K^\circ$, it bounds a
properly embedded oriented surface $P_i\subset Y_K^\circ$.  Near the
boundary we take
\[
P_i\cap(\partial Y_K^\circ\times[0,\varepsilon))
=
\gamma_i\times[0,\varepsilon).
\]
Use the same relative surface in both halves of the identity double and
glue the two copies along their boundary.  The resulting closed oriented
surface represents a class $V_i$.

Place the rim tori $R_1,R_2$ at different neck levels, so they are
disjoint even though $\gamma_1\cdot\gamma_2=1$.  In oriented local
coordinates on the neck,
\[
\Sigma_2\times S^1_\mu\times[-1,1],
\]
the part of $V_j$ crossing the neck is
$\gamma_j\times\{\mu_0\}\times[-1,1]$.  Hence
\[
R_i\cdot V_j=\gamma_i\cdot\gamma_j.
\]
In particular, after fixing orientations,
\[
R_1\cdot V_2=1,\qquad
R_2\cdot V_1=-1,\qquad
R_1\cdot V_1=R_2\cdot V_2=0.
\]

The doubled surface $V_i$ has even self-intersection.  Fix the boundary
normal framing coming from the product neck.  The two halves are copies
of the same relative surface, so their relative normal Euler numbers
agree modulo two; equivalently, the mod--two Euler number of the doubled
normal bundle is the sum of two identical relative classes.  Hence
$V_i^2$ is even.  Write
\[
V_1^2=2m_1,\qquad
V_2^2=2m_2,\qquad
V_1\cdot V_2=n.
\]
Set
\[
B_1=V_2-m_2R_1+nR_2,\qquad
B_2=-V_1-m_1R_2.
\]
A direct calculation gives
\[
R_i\cdot B_j=\delta_{ij},\qquad
B_1^2=B_2^2=B_1\cdot B_2=0.
\]
The classes $B_i$ may be represented by embedded oriented surfaces by
tubing the corresponding $V_i$ to parallel copies of the rim tori and
resolving the intersections inside the neck. Thus
\[
\langle R_1,B_1,R_2,B_2\rangle\cong2H.
\]
\end{proof}

\begin{proposition}\label{prop:id-double-H2-basis}
The identity double $X_K^{\mathrm{id}}$ has intersection form
\[
Q_{X_K^{\mathrm{id}}}\cong3H.
\]
More precisely, it contains two embedded genus-two surfaces $S,T$ with
\[
S^2=T^2=0,\qquad S\cdot T=1,
\]
and two disjoint essential square-zero rim tori $R_1,R_2$ which generate
the isotropic half of the remaining $2H$ summand.
\end{proposition}

\begin{proof}
Take $S$ to be a parallel copy of the gluing surface $\Sigma_2$.  In each
copy of $Y_K$ choose a parallel torus fiber $F_i'$ meeting $\Sigma_2$
once.  Removing a small disk about that intersection point gives a
punctured torus in $Y_i^\circ$, and the identity gluing joins the two
punctured tori to an embedded genus-two surface
\[
T=(F_1'\setminus D^2)\cup(F_2'\setminus D^2).
\]
The product framings give
\[
S^2=T^2=0,\qquad S\cdot T=1.
\]

Choose the intersection point $F_i'\cap\Sigma_2$ away from the curves
$\gamma_1,\gamma_2,d,y$.  The surfaces $S,T$ can then be arranged
disjoint from the two rim tori and from the neck portions used in
Lemma~\ref{lem:rim-dual-lattice}; hence the hyperbolic pair
$\langle S,T\rangle$ is orthogonal to the rank-four lattice there.
Consequently
\[
H\oplus2H
\]
embeds as a unimodular rank-six sublattice of
$H_2(X_K^{\mathrm{id}};\mathbb Z)$.

By Lemma~\ref{lem:id-double-H1},
\[
b_1(X_K^{\mathrm{id}})=2,\qquad e(X_K^{\mathrm{id}})=4.
\]
Poincare duality gives
\[
b_2(X_K^{\mathrm{id}})=e-2+2b_1=6.
\]
The displayed rank-six unimodular sublattice is therefore the whole
intersection lattice, proving
\[
Q_{X_K^{\mathrm{id}}}\cong3H.
\]
\end{proof}

The replacement tori
\[
R_d=d\times\mu,\qquad R_y=y\times\mu
\]
are chosen at collar levels disjoint from $R_1,R_2$ and from the neck
representatives of the dual classes.  Since the standard genus-two basis
pairs $(\gamma_1,\gamma_2)$ and $(d,y)$ on different handles, the curves
$\gamma_i$ are disjoint from $d$ and $y$.  The surfaces used above may
therefore be chosen disjoint from $\nu R_d\cup\nu R_y$.  The full
$3H$-lattice survives the two marked replacements.

Let $C_1,C_2$ be admissibly marked $\mathbb Z$--torus exteriors with
markings
\[
(\mu_i,\alpha_i,\beta_i),
\qquad
\pi_1(C_i)=\langle\mu_i\rangle\cong\mathbb Z,
\qquad
\alpha_i=\beta_i=1.
\]
On $\partial\nu R_y$ use the ordered basis
\[
(y,\mu_y,\rho_y),
\]
and on $\partial\nu R_d$ use
\[
(d,\mu_d,\rho_d).
\]
Choose the two boundary identifications
\[
y\mapsto\alpha_1,\qquad
\rho_y\mapsto\beta_1,\qquad
\mu_y\mapsto\mu_1,
\]
and
\[
d\mapsto\alpha_2,\qquad
\rho_d\mapsto\beta_2,\qquad
\mu_d\mapsto\mu_2.
\]
Denote the resulting closed manifold by
\[
\widehat X_K^{\mathrm{id}}(C_1,C_2).
\]

\begin{lemma}\label{lem:replacement-quotient}
Let $M$ be a connected $4$--manifold containing pairwise disjoint
square-zero tori $R_1,\ldots,R_r$.  Put
\[
W=M\setminus\Int\bigl(\nu R_1\cup\cdots\cup\nu R_r\bigr),
\]
and on $\partial\nu R_i$ choose an ordered basis
\[
(\lambda_i,\eta_i,\rho_i),
\]
where $\rho_i$ is the meridian of $R_i$.  Let
$(C_i;\mu_i,\alpha_i,\beta_i)$ be admissibly marked
$\mathbb Z$--torus exteriors and glue them so that
\[
\lambda_i\mapsto\alpha_i,\qquad
\rho_i\mapsto\beta_i,\qquad
\eta_i\mapsto\mu_i.
\]
Then
\[
\pi_1\!\left(
W\cup_{\partial\nu R_1}C_1\cup\cdots\cup_{\partial\nu R_r}C_r
\right)
\cong
\pi_1(M)\big/\!\big\langle\!\big\langle
\lambda_1,\ldots,\lambda_r
\big\rangle\!\big\rangle.
\]
\end{lemma}

\begin{proof}
Removing the tori introduces meridians $\rho_i$.  Restoring the deleted
neighborhoods $T^2\times D^2$ kills precisely these meridians, so van
Kampen gives
\[
\pi_1(W)\big/\!\big\langle\!\big\langle
\rho_1,\ldots,\rho_r
\big\rangle\!\big\rangle
\cong
\pi_1(M).
\]
For an admissibly marked exterior,
\[
\pi_1(C_i)=\langle\mu_i\rangle\cong\mathbb Z,
\qquad
\alpha_i=\beta_i=1.
\]
The gluing therefore imposes
\[
\lambda_i=1,\qquad \rho_i=1,\qquad \eta_i=\mu_i.
\]
The last relation introduces no new generator: $\mu_i$ is identified with
the already existing element $\eta_i$.  Eliminating the $\mu_i$ and then
the meridians $\rho_i$ leaves exactly
\[
\pi_1(M)\big/\!\big\langle\!\big\langle
\lambda_1,\ldots,\lambda_r
\big\rangle\!\big\rangle.
\]
\end{proof}

\begin{lemma}\label{lem:identity-double-group}
For the trefoil,
\[
\pi_1\bigl(\widehat X_K^{\mathrm{id}}(C_1,C_2)\bigr)=1.
\]
\end{lemma}

\begin{proof}
Apply Lemma~\ref{lem:replacement-quotient} to
\[
M=X_K^{\mathrm{id}}
\]
and to the two rim tori $R_y,R_d$.  With the chosen boundary markings,
$\lambda_1=y$ and $\lambda_2=d$.  Hence
\[
\pi_1\bigl(\widehat X_K^{\mathrm{id}}(C_1,C_2)\bigr)
\cong
\pi_1(X_K^{\mathrm{id}})
\big/\!\big\langle\!\big\langle y,d\big\rangle\!\big\rangle.
\]

Write
\[
G=\pi_1(Y_K\setminus\Int\nu\Sigma_2).
\]
The identity double is an amalgam of two copies of $G$ over
$\pi_1(\Sigma_2\times S^1)$.  In the first copy the four surface
generators map to
\[
a^{-1}b,\qquad b^{-1}aba^{-1},\qquad d,\qquad y,
\]
and the normal circle maps to the meridian $[x,b]$
\cite[Lemmas~4.6--4.7]{Akh07}.  In the second copy we use primed
generators, and the product gluing identifies the corresponding surface
generators and the two normal circles with opposite orientation.

Now quotient by the normal closure of $d$ and $y$.  The product gluing
also gives
\[
d'=y'=1.
\]
We claim that each copy of $G$ becomes trivial.  The presentation in
\cite[Lemma~4.7]{Akh07} contains
\[
aba=bab,\qquad
dxd^{-1}=xb,\qquad
dbd^{-1}=x^{-1}.
\]
Setting $d=1$ gives
\[
x=xb,
\]
hence $b=1$.  The third relation then gives
\[
1=x^{-1},
\]
so $x=1$.  With $b=1$, the braid relation becomes
\[
a^2=a,
\]
and hence $a=1$.

The meridian $[x,b]$ is therefore trivial.  Lemma~4.7 states that the
additional generators $g_i$ and relators $r_j$ lie in the normal subgroup
generated by $[x,b]$; hence the $g_i$ also become trivial.  Its remaining
relator $r_{n+1}$ becomes $[x,a]=1$, which is already satisfied.  Thus
\[
G/\!\langle\!\langle d,y\rangle\!\rangle=1.
\]
The identical argument applies to the primed copy.  Since both vertex
groups of the amalgam become trivial,
\[
\pi_1\bigl(\widehat X_K^{\mathrm{id}}(C_1,C_2)\bigr)=1.
\]
\end{proof}

\begin{theorem}\label{thm:id-rank6}
For any two admissibly marked $\mathbb Z$--torus exteriors,
\[
\pi_1\bigl(\widehat X_K^{\mathrm{id}}(C_1,C_2)\bigr)=1,
\qquad
e\bigl(\widehat X_K^{\mathrm{id}}(C_1,C_2)\bigr)=8,
\qquad
\sigma\bigl(\widehat X_K^{\mathrm{id}}(C_1,C_2)\bigr)=0.
\]
Moreover
\[
Q_{\widehat X_K^{\mathrm{id}}}\cong3H,
\]
and therefore
\[
\widehat X_K^{\mathrm{id}}(C_1,C_2)
\approx
\#_3(S^2\times S^2).
\]
\end{theorem}

\begin{proof}
Simple connectivity is Lemma~\ref{lem:identity-double-group}.  Each
replacement removes $T^2\times D^2$ and inserts a torus exterior $C_i$
with
\[
e(C_i)=2,\qquad \sigma(C_i)=0.
\]
Hence
\[
e=4+2+2=8,\qquad \sigma=0.
\]
By Proposition~\ref{prop:id-double-H2-basis} and the paragraph following
it, the full $3H$-lattice is represented away from the replacement
regions and survives in the final manifold.  Since simple connectivity and
$e=8$ give $b_2=6$, this is the full intersection form.  Freedman's theorem \cite{Freedman} gives the stated
homeomorphism type.
\end{proof}

\paragraph{Boyle and Matsumoto--Fukaya choices.}

Taking $C_i$ to be Boyle exteriors gives the Boyle version of
Theorem~\ref{thm:id-rank6}.  Taking $C_i=C_{\mathrm{MF}}$ with the marking
of Proposition~\ref{prop:MF-complement} gives the
Matsumoto--Fukaya version.  The same van Kampen calculation applies because
only the marked peripheral data are used.

\section{Intersection forms and symplectic consequences}

We use the standard symplectic framework originating in Gromov's work on
pseudoholomorphic curves \cite{Gromov85}; the symplectic sum operation used
here is the construction of Gompf and McCarthy--Wolfson
\cite{Gompf95,McCarthyWolfson94}.
\label{sec:geometric-symplectic}
\paragraph{Relation with the earlier constructions.}

The symplectic blocks $Y_K$ and $X_K$ were introduced in \cite{Akh07}; the
fundamental-group calculations used here appear in \cite{AP08}. The
Luttinger-surgery presentation of the constituent $M_K\times S^1$ was recalled
in Subsection~\ref{subsec:MK-luttinger}. In particular, when $g=1$, replacing
the mapping-torus description of $M_K\times S^1$ by two Luttinger surgeries
on $T^4$ changes the presentation, not the underlying building block.
Luttinger surgery was then used systematically in \cite{AP08,AP10} to impose
additional relations on Lagrangian product tori while preserving
symplecticity. Related torus-surgery approaches to small exotic
$4$-manifolds include reverse engineering and surgery on nullhomologous
tori; see \cite{FPSreverse,FSnull}.

The knot-surgery and Luttinger-surgery viewpoints are retained throughout
because they illuminate complementary parts of the construction. For
$M_K\times S^1$ the relation is exact: the monodromy of the fibered knot
determines a Luttinger-surgery presentation. The knot-surgery/fiber-sum
description of $Y_K$ and $X_K$ then makes the surviving section, fiber,
genus-two surface, and rim tori visible. The higher product-surface
Luttinger constructions show how analogous fundamental-group relations can
be imposed symplectically on standard products. We do not identify those
higher comparison models with a specified $Y_K$ or $X_K$ unless the cited
construction proves that identification.

The replacement considered here is different. A Luttinger surgery removes
and reglues $T^2\times D^2$, so it preserves Euler characteristic and
signature. Here a rim-torus neighborhood is replaced by a torus exterior
$C_T\subset S^4$ with
\[
e(C_T)=2.
\]
Accordingly the Euler characteristic changes by two for each inserted
exterior. This distinction separates the Luttinger-surgery presentation of
the starting blocks from the marked-exterior replacement studied below.
\subsection{Homology and geometric representatives}
\label{sec:homology-representatives}

We identify the second-homology classes before and after the Boyle replacements.

\paragraph{The homology of $Y_K$.}

Abelianize the presentation of $\pi_1(Y_K)$ from
\cite[Lemma~7]{AP08}. In additive notation, the braid relation gives $a=b$.
The relation
\[
dxd^{-1}=xb
\]
gives $b=0$, hence $a=0$. The relation
\[
dbd^{-1}=x^{-1}
\]
then gives $x=0$. The final longitude relation contributes nothing further in
abelianization. Therefore only $d$ and $y$ remain.

\begin{proposition}\label{prop:H1YK}
For the trefoil block,
\[
H_1(Y_K;\mathbb Z)\cong
\mathbb Z\langle d\rangle\oplus\mathbb Z\langle y\rangle .
\]
Consequently
\[
H_3(Y_K;\mathbb Z)\cong\mathbb Z^2.
\]
Since $e(Y_K)=0$,
\[
H_2(Y_K;\mathbb Z)\cong\mathbb Z^2.
\]
\end{proposition}

\begin{proof}
The calculation of $H_1$ is the abelianization above. Poincare duality gives
$b_3=b_1=2$. Since
\[
0=e(Y_K)=1-b_1+b_2-b_3+1,
\]
we obtain $b_2=2$. Because $H_1(Y_K)$ is torsion free, the universal
coefficient theorem together with Poincare duality shows that $H_2(Y_K)$ is
torsion free.
\end{proof}

\paragraph{A geometric basis of $H_2(Y_K)$.}

Recall from \cite[Section~4]{Akh07} that
\[
Y_K=(M_K\times S^1)\#_{F=T_m}(M_K\times S^1).
\]
Let $T_1$ denote the section in the first copy and $T_2$ a fiber in the second
copy. Their punctured versions glue to the genus-two surface
\[
\Sigma_2=T_1\#T_2\subset Y_K.
\]
The source states that $\Sigma_2$ is symplectically embedded and has
self-intersection zero.

Choose a second fiber
\[
F'\subset M_K\times S^1
\]
in the first copy, parallel to the fiber removed in forming $Y_K$ and disjoint
from that removed fiber. The source states explicitly that torus fibers and
torus sections in $M_K\times S^1$ are symplectically embedded and have
self-intersection zero \cite[Section~4]{Akh07}. Choosing the parallel fiber
inside the symplectic torus-bundle region, $F'$ survives as a closed square-zero
torus in $Y_K$. We use only its smooth embedded representative and intersection
numbers below; no later argument depends on $F'$ remaining symplectic after the
Boyle replacements.

\begin{proposition}\label{prop:H2YKbasis}
The classes
\[
[\Sigma_2],\ [F']\in H_2(Y_K;\mathbb Z)
\]
form a basis. Their intersection matrix is
\[
H=
\begin{pmatrix}
0&1\\
1&0
\end{pmatrix}.
\]
Hence
\[
Q_{Y_K}\cong H.
\]
\end{proposition}

\begin{proof}
We have
\[
\Sigma_2^2=(F')^2=0.
\]
The section $T_1$ meets each fiber of the first torus bundle once. Since $F'$
is disjoint from the fiber removed in forming $Y_K$, its intersection point
with $T_1$ survives the fiber sum. The second half of $\Sigma_2$ lies in the
other summand and is disjoint from $F'$. Hence
\[
\Sigma_2\cdot F'=1.
\]
The two classes are therefore linearly independent. Proposition~\ref{prop:H1YK}
shows that $H_2(Y_K;\mathbb Z)$ has rank two, so they form a basis.
\end{proof}

\begin{corollary}
The manifold $Y_K$ is spin.
\end{corollary}

\begin{proof}
Its intersection form is even.
\end{proof}

\paragraph{The final manifold $Z_K(B_1,B_2)$.}

The two rim tori $R_y,R_d$ lie in a collar of
$\partial\nu\Sigma_2$. The core surface $\Sigma_2$ is disjoint from this
collar. The parallel fiber $F'$ can be chosen so that its unique intersection
point with $\Sigma_2$ lies away from $y\cup d$. Therefore both $\Sigma_2$
and $F'$ are disjoint from the two torus neighborhoods removed in the Boyle
replacement.

Consequently the same embedded surfaces survive in
\[
Z_K(B_1,B_2)
\]
and still satisfy
\[
\Sigma_2^2=(F')^2=0,\qquad
\Sigma_2\cdot F'=1.
\]
Since
\[
H_2(Z_K;\mathbb Z)\cong\mathbb Z^2,
\]
they form a basis of the full second homology:
\[
H_2(Z_K;\mathbb Z)
=
\mathbb Z[\Sigma_2]\oplus\mathbb Z[F'].
\]

Thus the homeomorphism theorem is represented geometrically by an explicit
genus-two surface and an explicit torus.

\paragraph{Geometric representatives and genus bounds.}
\label{sec:minimal-genus}

For the $Y_K$ replacement the second homology is generated by the two surviving
classes
\[
A=[\Sigma_2],\qquad B=[F'],
\]
with
\[
A^2=B^2=0,\qquad A\cdot B=1.
\]
The construction supplies an embedded genus-two representative of $A$ and an
embedded torus representative of $B$. Consequently
\[
g_{\min}(A)\le2,\qquad g_{\min}(B)\le1.
\]

\paragraph{The dual genus-two basis in the original $X_K$.}

The original symplectic construction of $X_K$ contains a useful comparison
configuration. Akhmedov \cite[Lemma~4.1]{Akh07} identifies a basis
\[
[S],[T]\in H_2(X_K;\mathbb Z)
\]
represented by embedded genus-two surfaces satisfying
\[
S^2=T^2=0,\qquad S\cdot T=1.
\]
Here $S$ is a parallel copy of the genus-two surface used in the second
fiber sum, while $T$ is the new genus-two surface obtained by gluing two
punctured genus-one surfaces. Thus
\[
Q_{X_K}\cong H.
\]
The same paper computes
\[
K_{X_K}=2S+2T,
\qquad
K_{X_K}^2=8.
\]

This is a concrete dual genus-two configuration in the original
symplectic cohomology $S^2\times S^2$. By contrast, the displayed basis
$([\Sigma_2],[F'])$ in the present $Y_K$ replacement has representatives
of genera $(2,1)$.

\subsection{Symplectic and adjunction consequences}
\label{sec:symplectic-status}

The symplectic information from the original knot-surgery blocks does not
automatically pass through the replacement.

Akhmedov \cite[Section~4]{Akh07} states that, for a genus-one fibered knot $K$,
the manifold
\[
M_K\times S^1
\]
is symplectic. Both its torus fiber and torus section are symplectically
embedded and have self-intersection zero. Gompf's theorem therefore gives a
symplectic structure on $Y_K$, and the genus-two surface $\Sigma_2$ is
symplectic. The parallel fiber $F'$ may also be chosen symplectic. A second
Gompf sum gives the symplectic manifold $X_K$. The same paper states, using
Usher's minimality theorem, that both $Y_K$ and $X_K$ are minimal symplectic
manifolds. These statements are part of the original construction
\cite[Section~4]{Akh07}.

\begin{proposition}\label{prop:symplectic-status}
For a genus-one fibered knot $K$:
\begin{enumerate}
\item $M_K\times S^1$, $Y_K$, and $X_K$ are symplectic;
\item $Y_K$ and $X_K$ are minimal symplectic;
\item in $Y_K$, the genus-two surface $\Sigma_2=T_1\#T_2$ is symplectic,
  and a parallel torus fiber $F'$ can be chosen symplectic;
\item the present argument does not prove that $Z_K(B_1,B_2)$ is symplectic.
\end{enumerate}
\end{proposition}

\begin{proof}
The first three statements follow from the torus-bundle geometry and
Gompf-sum construction in \cite[Section~4]{Akh07}; minimality is stated
there using Usher's theorem. No symplectic extension over the marked
Boyle exteriors is constructed here.
\end{proof}

The replacements leave neighborhoods of $\Sigma_2$ and $F'$ untouched,
but extending the original symplectic form across the Boyle exteriors is a
global problem. No such extension is established here.

The Luttinger-surgery constructions of \cite{AP08,AP10} provide genuine
symplectic comparison examples because Luttinger surgery on a Lagrangian torus
preserves symplecticity. This does not imply that the Boyle replacement is
symplectic.

\paragraph{An adjunction obstruction.}
\label{sec:adjunction-obstruction}

The geometric basis of $H_2(Y_K)$ gives more information than the parity of the
intersection form. It also determines the canonical class of the symplectic structure
used in the original knot-surgery construction.

Let $\omega_Y$ denote the symplectic form on $Y_K$ supplied by Gompf's symplectic
sum construction. As recalled above, the genus-two surface $\Sigma_2$ and the
parallel torus fiber $F'$ can be chosen symplectic for this structure.

\begin{proposition}\label{prop:KY}
With respect to the basis
\[
H_2(Y_K;\mathbb Z)=
\mathbb Z[\Sigma_2]\oplus\mathbb Z[F'],
\]
the Poincare dual of the symplectic canonical class is
\[
PD(K_{Y_K})=2[F'].
\]
In particular,
\[
K_{Y_K}^2=0.
\]
\end{proposition}

\begin{proof}
For an embedded symplectic surface $C$ in a symplectic $4$-manifold, the adjunction
formula gives
\[
2g(C)-2=C^2+K\cdot C.
\]
Applying this to $\Sigma_2$, which has genus two and square zero, gives
\[
K_{Y_K}\cdot\Sigma_2=2.
\]
Applying it to the square-zero symplectic torus $F'$ gives
\[
K_{Y_K}\cdot F'=0.
\]

Write
\[
PD(K_{Y_K})=A[\Sigma_2]+B[F'].
\]
Since the intersection matrix in this basis is $H$,
\[
K_{Y_K}\cdot F'=A,\qquad
K_{Y_K}\cdot\Sigma_2=B.
\]
Thus $A=0$ and $B=2$, proving the formula. Its square is zero. This agrees with
the independent calculation $c_1^2(Y_K)=0$ in \cite[Section~4]{Akh07}.
\end{proof}

The calculation obstructs the direct attempt to make the
two-Boyle replacement symplectic.

\begin{theorem}\label{thm:no-extension}
There is no symplectic form on $Z_K(B_1,B_2)$ for which the two surviving
embedded surfaces $\Sigma_2$ and $F'$ are both symplectic with the orientations
for which
\[
\Sigma_2\cdot F'=1.
\]
Thus the original fiber-sum symplectic form on the unchanged region cannot
extend over the two Boyle exteriors while keeping both surfaces symplectic.
\end{theorem}

\begin{proof}
Suppose such a symplectic form existed. The two surfaces would satisfy
\[
\Sigma_2^2=(F')^2=0,\qquad
\Sigma_2\cdot F'=1.
\]
The adjunction formula would therefore give
\[
K_{Z_K}\cdot\Sigma_2=2,\qquad
K_{Z_K}\cdot F'=0.
\]
Since these two classes form a basis of $H_2(Z_K;\mathbb Z)$, the same calculation as
in Proposition~\ref{prop:KY} would give
\[
PD(K_{Z_K})=2[F']
\]
and hence
\[
K_{Z_K}^2=0.
\]

On the other hand, every closed almost-complex $4$-manifold satisfies
\[
c_1^2=2e+3\sigma.
\]
For $Z_K$ we have $e=4$ and $\sigma=0$, so any symplectic structure would satisfy
\[
K_{Z_K}^2=c_1^2(Z_K)=8,
\]
a contradiction.
\end{proof}

\begin{remark}
Theorem~\ref{thm:no-extension} only obstructs symplectic structures for
which the displayed surfaces $\Sigma_2$ and $F'$ are simultaneously
symplectic. It does not rule out other symplectic structures on
$Z_K(B_1,B_2)$.
\end{remark}

\paragraph{Dual symplectic generators.}
\label{sec:dual-genus-two}

\begin{theorem}\label{thm:dual-genus-two}
Let $(X,\omega)$ be a closed symplectic $4$-manifold homeomorphic to
$S^2\times S^2$. Suppose that a hyperbolic basis
\[
A,B\in H_2(X;\mathbb Z),\qquad
A^2=B^2=0,\qquad A\cdot B=1,
\]
is represented by embedded symplectic surfaces $C_A,C_B$. Then
\[
\bigl(g(C_A),g(C_B)\bigr)=(0,0)\quad\text{or}\quad(2,2).
\]
If $X$ is not diffeomorphic to the standard $S^2\times S^2$, then necessarily
\[
g(C_A)=g(C_B)=2.
\]
\end{theorem}

\begin{proof}
Write
\[
PD(K_\omega)=aA+bB
\]
for the symplectic canonical class. Since $X$ has
\[
e(X)=4,\qquad \sigma(X)=0,
\]
we have
\[
K_\omega^2=c_1^2(X)=2e(X)+3\sigma(X)=8.
\]
Because the intersection form is $H$,
\[
K_\omega^2=2ab,
\]
so
\[
ab=4.
\]

The class $K_\omega$ is characteristic. For the even form $H$, every
characteristic element has even coefficients, hence $a$ and $b$ are even.
Adjunction for the two symplectic surfaces gives
\[
2g(C_A)-2=K_\omega\cdot A=b,
\qquad
2g(C_B)-2=K_\omega\cdot B=a.
\]
Therefore $a,b\ge-2$. The only even integer solutions of
\[
ab=4,\qquad a,b\ge-2
\]
are
\[
(a,b)=(2,2)\quad\text{and}\quad(a,b)=(-2,-2).
\]
These give respectively
\[
(g(C_A),g(C_B))=(2,2)
\quad\text{and}\quad
(g(C_A),g(C_B))=(0,0).
\]

In the second case $X$ contains a symplectically embedded sphere of square zero.
By McDuff's classification of symplectic $4$-manifolds containing a symplectic sphere
of nonnegative self-intersection \cite{McDuff90}, $X$ is rational or ruled.
Since $X$ is simply connected, has $b_2=2$, and has even intersection form $H$,
its underlying smooth manifold is the trivial $S^2$-bundle over $S^2$.
Hence it is diffeomorphic to the standard $S^2\times S^2$.

Thus an exotic smooth structure admitting such a dual symplectic basis can only
occur in the $(2,2)$ case.
\end{proof}

\begin{remark}
The conclusion applies when a hyperbolic basis has symplectic
representatives. For $Z_K(B_1,B_2)$ the displayed basis
$(\Sigma_2,F')$ has genera $(2,1)$, so these two surfaces cannot be
simultaneously symplectic.
\end{remark}

\subsection{Minimal genus and a symplectic standardness criterion}
\label{sec:min-genus-symplectic}

The visible square-zero torus in the $S^2\times S^2$ replacement gives a
strong restriction on possible symplectic structures. We first isolate the
general $b_2^+=1$ statement.

\begin{proposition}\label{prop:min-genus-H}
Let $(X,\omega)$ be a closed symplectic $4$--manifold homeomorphic to
$S^2\times S^2$, and let $A\in H_2(X;\mathbb Z)$ be a primitive class with
\[
A^2=0.
\]
If the symplectic canonical class satisfies
\[
K_\omega\cdot[\omega]>0,
\]
then every smoothly embedded connected surface representing $A$ has genus at
least two:
\[
g_{\min}(A)\ge 2.
\]
\end{proposition}

\begin{proof}
Choose $B\in H_2(X;\mathbb Z)$ with
\[
A^2=B^2=0,\qquad A\cdot B=1.
\]
Since $Q_X\cong H$, every characteristic class has even coefficients in this
basis. Write
\[
PD(K_\omega)=2pA+2qB.
\]
The almost-complex identity gives
\[
K_\omega^2=2e(X)+3\sigma(X)=8,
\]
hence
\[
8pq=8,\qquad pq=1.
\]
Thus
\[
(p,q)=(1,1)\quad\text{or}\quad(-1,-1),
\]
and in particular
\[
|K_\omega\cdot A|=2.
\]

Because $b_2^+(X)=1$, the Seiberg--Witten chamber must be specified. Under
the hypothesis $K_\omega\cdot[\omega]>0$, the class $-K_\omega$ lies in the
Taubes chamber \cite{Taubes94,Taubes95} in which
\[
SW^-_X(-K_\omega)=\pm1.
\]
For $b_1(X)=0$ and $b_2^-(X)=1$, the $b_2^+=1$ generalized adjunction
inequality of Li--Liu \cite{LLadj} applies to any connected embedded surface
$\Sigma$ representing $A$ with $A^2\ge0$. Therefore
\[
2g(\Sigma)-2
\ge
A^2+\bigl|K_\omega\cdot A\bigr|
=2,
\]
so $g(\Sigma)\ge2$.
\end{proof}

\begin{theorem}\label{thm:symplectic-implies-standard}
Let $Z$ be any of the manifolds obtained in Case II by two admissibly marked
replacements of the trefoil block $Y_K$. In particular, $Z$ may be a Boyle
or a Matsumoto--Fukaya replacement. If $Z$ admits a symplectic structure,
then
\[
\boxed{Z\cong S^2\times S^2.}
\]
Consequently, if such a $Z$ is exotic, then it is nonsymplectic.
\end{theorem}

\begin{proof}
By Theorem~\ref{thm:general-replacement},
\[
\pi_1(Z)=1,\qquad Q_Z\cong H,
\]
so $Z$ is homeomorphic to $S^2\times S^2$. Moreover, the construction
contains a square-zero torus $F'$ and a genus-two surface $\Sigma_2$ with
\[
(F')^2=\Sigma_2^2=0,\qquad
F'\cdot\Sigma_2=1.
\]
Hence $[F']$ is primitive.

Assume that $(Z,\omega)$ is symplectic and let $K_\omega$ be its canonical
class. If
\[
K_\omega\cdot[\omega]>0,
\]
Proposition~\ref{prop:min-genus-H} applied to $A=[F']$ says that every
embedded representative of $[F']$ has genus at least two. This contradicts
the embedded torus $F'$. Therefore
\[
K_\omega\cdot[\omega]\le0.
\]
Since
\[
K_\omega^2=2e(Z)+3\sigma(Z)=8>0
\]
and $[\omega]^2>0$ in a vector space of signature $(1,1)$,
$K_\omega\cdot[\omega]$ cannot vanish. Thus
\[
K_\omega\cdot[\omega]<0.
\]

By Liu's $b_2^+=1$ classification theorem \cite{Liu96}, a symplectic
$4$--manifold with $K_\omega\cdot[\omega]<0$ is rational or ruled. Since
$Z$ is simply connected and has even rank-two intersection form, the only
possible diffeomorphism type is
\[
S^2\times S^2.
\]
\end{proof}

\begin{corollary}\label{cor:exotic-nonsymplectic}
For the Case II family,
\[
Z\not\cong S^2\times S^2
\quad\Longrightarrow\quad
Z\ \text{admits no symplectic structure}.
\]
\end{corollary}

The primitive square-zero torus is essential in this argument.

\section{Seiberg--Witten invariants and smooth questions}\label{sec:SW}

The embedded surfaces constructed in Case II give a direct
Seiberg--Witten calculation. We use the conventions and adjunction
inequalities recalled in
\cite[Section~3]{Akh07}: Theorem~3.6 there is the generalized adjunction
inequality for the small-perturbation invariant when $b_2^+=1$, and
Theorem~3.5 is the usual adjunction inequality when $b_2^+>1$.

\subsection{The Case II invariant}

Let $Z$ be a Case II manifold. Thus
\[
\pi_1(Z)=1,\qquad Q_Z\cong H,
\]
and there are embedded surfaces
\[
A=F',\qquad B=\Sigma_2
\]
with
\[
A^2=B^2=0,\qquad A\cdot B=1,\qquad
g(A)=1,\qquad g(B)=2.
\]
Since $b_2^+(Z)=b_2^-(Z)=1$, we use the small-perturbation
Seiberg--Witten invariant $SW_Z^\circ$ of \cite{Szabo96}; its chamber
independence in the present range is recalled in
\cite[Theorem~3.2]{Akh07}.

\begin{theorem}\label{thm:SW-vanishing}
For every Case II manifold $Z$,
\[
\boxed{SW_Z^\circ\equiv0.}
\]
\end{theorem}

\begin{proof}
Assume that $k$ is a small-perturbation basic class. By
\cite[Theorem~3.6]{Akh07}, equivalently the Li--Liu adjunction inequality,
the essential square-zero torus $A$ satisfies
\[
0=2g(A)-2\ge |k\cdot A|.
\]
Thus
\[
k\cdot A=0.
\]
Since $A$ is primitive and $Q_Z=H$, its orthogonal complement in
$H_2(Z;\mathbb Z)$ is exactly $\mathbb Z A$. Hence
\[
PD(k)=mA
\]
for some integer $m$, and therefore
\[
k^2=0.
\]
On the other hand a nonzero Seiberg--Witten invariant can occur only in
nonnegative formal dimension, whereas
\[
d(k)
=
\frac{k^2-(2e(Z)+3\sigma(Z))}{4}
=
\frac{0-8}{4}
=-2.
\]
This is impossible. Hence $SW_Z^\circ$ vanishes identically.
\end{proof}

\paragraph{The identity-glued rank-six family.}

The geometric basis from Proposition~\ref{prop:id-double-H2-basis} also
determines the Seiberg--Witten invariant of the Case III replacements.

\begin{theorem}\label{thm:SW-caseIII}
For every identity-glued Case III manifold
\[
\widehat X_K^{\mathrm{id}}(C_1,C_2),
\]
the Seiberg--Witten invariant vanishes:
\[
SW_{\widehat X_K^{\mathrm{id}}}\equiv0.
\]
Consequently these manifolds are not symplectic.
\end{theorem}

\begin{proof}
Use the decomposition of the intersection lattice supplied by
Proposition~\ref{prop:id-double-H2-basis} and
Lemma~\ref{lem:rim-dual-lattice}:
\[
H_2=
\langle S,T\rangle\oplus
\langle R_1,B_1\rangle\oplus
\langle R_2,B_2\rangle,
\]
with each summand hyperbolic.  The classes $R_1,R_2$ are represented by
essential square-zero tori, while $S,T$ are represented by genus-two
surfaces.

Let $k$ be a basic class and write
\[
PD(k)=pS+qT+rR_1+sB_1+uR_2+vB_2.
\]
Since $b_2^+=3>1$, the usual adjunction inequality applies.  Applied to
the two tori it gives
\[
k\cdot R_1=s=0,\qquad
k\cdot R_2=v=0.
\]
Hence
\[
k^2=2pq.
\]
Adjunction on $S$ and $T$ gives
\[
|k\cdot S|=|q|\le2,\qquad
|k\cdot T|=|p|\le2.
\]
Therefore
\[
k^2=2pq\le8.
\]
For the Case III manifolds,
\[
2e+3\sigma=16,
\]
so the formal dimension is
\[
d(k)=\frac{k^2-16}{4}\le-2.
\]
Since a basic class has nonnegative formal dimension, there are no
basic classes. Hence $SW\equiv0$.

If the manifold were symplectic, Taubes' nonvanishing theorem
\cite{Taubes94,Taubes95,Taubes96} would give a nonzero canonical basic
class because $b_2^+>1$, a contradiction.
\end{proof}

\subsection{Marked handle calculus and smooth standardness}
\label{sec:smooth-standardness}

The constructions above determine the relevant homeomorphism types, but
not their diffeomorphism types.  A handle-calculus approach would require
a marked description of the Boyle exterior in which the peripheral
classes
\[
(\mu_B,\alpha_B,\beta_B)
\]
and the prescribed boundary identifications are explicit.  The available
description does not provide this marked data in a form sufficient for the
closed handle calculation.  Thus smooth standardness of the homotopy
spheres and of the $S^2\times S^2$ examples remains open.

For the identity-glued rank-six family, the geometric basis gives
intersection form $3H$ and hence the topological type
$\#_3(S^2\times S^2)$.  Its smooth type is not determined here.

\begin{remark}[On standardness]
The Seiberg--Witten calculations above are consistent with the possibility
that some of the manifolds constructed here are standard rather than exotic.
In fact, it seems quite possible that this occurs for some of the examples.
A careful analysis of the corresponding handlebody diagrams may settle the
question. Related handle-calculus phenomena for cusp and fishtail neighborhoods,
and for logarithmic transforms on $S^2\times S^2$, appear in
\cite{Akbulut99,Akbulut13Log}.
The smooth classification of these manifolds, including the marked handle-calculus problem above, will be pursued in future work.
\end{remark}

\section*{Acknowledgments}

My interest in the questions considered here originated in conversations with Cliff Taubes, and I am grateful to him for several very helpful discussions. I also thank Çağrı Karakurt and B. Doug Park for several enjoyable and stimulating conversations about these four-manifolds and related questions. An LLM-based tool was used to assist with the preparation of portions of the text, including grammar and with the generation of the figures. All mathematical content is due to the author, who takes full intellectual responsibility for the content of this paper.

\end{document}